\documentclass[a4paper]{amsart}

\usepackage[all, knot]{xy}
\usepackage{enumerate}

\newcommand{\myemph}{\emph}

\newcommand{\defeq}{=_{\mathrm{def}}}

\newcommand{\Set}{\mathbf{Set}}
\newcommand{\Cat}{\mathbf{Cat}}

\newcommand{\Grph}{\mathbf{Grph}}
\newcommand{\Poly}{\mathbf{Poly}}
\newcommand{\PolyMnd}{\mathbf{PolyMnd}}
\newcommand{\PolyEnd}{\mathbf{PolyEnd}}

\newcommand{\Sp}{\mathbf{Span}}
\newcommand{\Span}{\mathbf{Span}}

\newcommand{\bbC}{\mathbb{C}}

\newcommand{\catE}{\mathcal{E}}
\newcommand{\catK}{\mathcal{K}}

\newcommand{\Obj}{\mathrm{Obj}}

\newcommand{\Hor}{\mathrm{Hor}}
\newcommand{\Ver}{\mathrm{Ver}}

\newcommand{\Sq}{\mathrm{Sq}}

\newcommand{\iso}{\cong}

\newcommand{\End}{\mathrm{End}}
\newcommand{\Mnd}{\mathrm{Mnd}}

\newcommand{\Und}{\mathrm{Und}}
\newcommand{\Inc}{\mathrm{Inc}}

\newcommand{\Qalg}{Q\text{-}\mathrm{alg}}
\newcommand{\QAlg}{Q\text{-}\mathrm{Alg}}
\newcommand{\QQalg}{Q^*\text{-}\mathrm{Alg}}

\newcommand{\ladj}[1] {{#1}_{!}}
\newcommand{\radj}[1] {{#1}_{\ast}}
\newcommand{\pbk}[1] {{#1}^{\ast}}

\SelectTips{cm}{}
\newdir{ >}{{}*!/-7pt/@{>}}
\newcommand{\xycenter}[1]{\vcenter{\hbox{\xymatrix{#1}}}}

\newcommand{\drpullback}[1][dr]{\save*!/#1-1.2pc/#1:(-1,1)@^{|-}\restore}

\newtheorem{thm}{Theorem}[section]

\newtheorem{lem}[thm]{Lemma}

\newtheorem{prop}[thm]{Proposition}

\theoremstyle{definition}
\newtheorem{defn}[thm]{Definition}

\newtheorem{rmk}[thm]{Remark}
\newtheorem{examp}[thm]{Example}

\begin{document}

\title{Monads in Double Categories}

\author[T. M. Fiore]{Thomas M. Fiore}
\address{Department of Mathematics and Statistics, University of
Michigan-Dearborn, 4901 Evergreen Road, Dearborn, MI 48128, USA}
\email{tmfiore@umd.umich.edu}

\author[N. Gambino]{Nicola Gambino}
\address{Dipartimento di Matematica e Informatica, via Archirafi 34, 90123 Palermo, Italy \\ and 
School of Mathematics, The University of Manchester, Oxford Road, Manchester M13 9PL, UK}
\email{ngambino@math.unipa.it}

\author[J. Kock]{Joachim Kock}
\address{Departament de Matem\`atiques,
Universitat Aut\`onoma de Barcelona, 08193 Bellaterra (Barcelona),
Spain} \email{kock@mat.uab.cat}



\begin{abstract}
We extend the basic concepts of Street's formal theory of monads from the setting of 2-categories to that
of double categories. In particular, we introduce the double category $\Mnd(\bbC)$ of monads in a double
category~$\bbC$ and define what it means for a double category to admit the construction of free monads.
Our main theorem shows that, under some mild conditions, a double category that is a framed bicategory
admits the construction of free monads if its horizontal 2-category does. We apply this result to
obtain double adjunctions which extend the adjunction between graphs and categories and the adjunction
between polynomial endofunctors and polynomial monads.
\end{abstract}


\date{July 11th, 2010}

\maketitle

\section*{Introduction}

The development of the formal theory of monads, begun in~\cite{StreetR:fortm} and continued
in~\cite{LackS:fortm}, shows that much of the theory of monads~\cite{BarrM:toptt}
can be generalized from the setting of the 2-category~$\Cat$ of small categories, functors and natural
transformations to that of a general 2-category. The generalization, which involves defining the
2-category~$\Mnd(\catK)$ of monads, monad maps and monad 2-cells in a 2-category~$\catK$,
is useful to study homogeneously a variety of important
mathematical structures. For example, as explained in~\cite{LeinsterT:higohc},
categories, operads,  multicategories and~$T$-multicategories can
all be seen as monads in appropriate bicategories. However, the most
natural notions of a morphism between these mathematical structures do not appear as instances of the notion
of a monad map. For example, it is well known that, while
categories can be viewed as monads in the bicategory of
spans~\cite{BenabouJ:intb}, functors are not monad maps
therein.

To address this issue, we define the double category~$\Mnd(\bbC)$
of monads, horizontal monad maps, vertical monad maps and
monad squares in a double category~$\bbC$. Monads and horizontal
monad maps in~$\bbC$ are exactly monads and monad maps
in the horizontal 2-category of~$\bbC$, while the definitions of vertical monad maps and
monad squares in~$\bbC$  involve vertical arrows of~$\bbC$ that are
not necessarily identities. This combination of horizontal and
vertical arrows of~$\bbC$ in the definition of~$\Mnd(\bbC)$ allows
us  to describe mathematical structures and morphisms between them
as monads and vertical monad maps  in appropriate double
categories. For example, small categories and functors can be viewed as
monads and vertical monad maps in the the double category of spans.

For a double category $\bbC$, we define also the double category~$\End(\bbC)$ of endomorphisms, horizontal endomorphism
maps, vertical endomorphism maps and endomorphism squares. The double categories $\Mnd(\bbC)$ and $\End(\bbC)$ are
related by a forgetful double functor~$U : \Mnd(\bbC) \rightarrow \End(\bbC)$, mapping a monad to its underlying
endomorphism. By definition, a double category~$\bbC$ is said to admit the construction of free monads if~$U$ has a
vertical left adjoint. In view of our applications, we consider the construction of free monads in double categories
that satisfy the additional assumption of being framed bicategories, in the sense of~\cite{ShulmanM:frabmf}. Our main result shows that a
framed bicategory satisfying some mild assumptions admits the construction of free monads if its
horizontal 2-category does. Here, the notion of a 2-category admitting the construction of free monads is obtained by
generalizing the characterization of the free monads in the 2-category~$\Cat$ obtained in~\cite[\S6.1]{StatonS:namppc}.

We apply the general theory to the study of two free constructions. First, we consider the construction of the free
category on a graph (relatively to a category with finite limits), which plays an important role in Joyal's abstract
treatment of G\"odel's incompleteness theorems~\cite{MaiettiM:joyaul}. We show that if $\catE$ is a pretopos with
parametrized list objects, then the double category of spans in $\catE$ admits the construction of free monads. Secondly,
we consider the construction of the free monad on a polynomial endofunctor (relatively to a locally cartesian closed
category, which is always assumed here to have a terminal object), which contributes to the category-theoretic analysis of Martin-L\"of's types of wellfounded trees,
begun in~\cite{MoerdijkI:weltc} and continued in~\cite{GambinoN:weltpf,GambinoN:polfpm}. We show that if $\catE$
is a locally cartesian closed category with finite disjoint coproducts and W-types, then the double category of
polynomials in $\catE$ admits the construction of free monads. Both of these results are obtained by
application of our main result, which is possible since the double categories of interest are framed bicategories.
Examples of categories $\catE$ satisfying the hypotheses above abound: for example, every elementary topos
with a natural numbers object is both a pretopos with parametrized list objects and a locally cartesian closed
category with finite disjoint coproducts and W-types~\cite{MoerdijkI:weltc}. Thus, our theory applies in particular
to the category~$\Set$ of sets and functions and to categories of sheaves.

The double categories of spans and of polynomials are defined so that if we consider the vertical part of the
free monad double adjunction, we recover exactly the adjunction between graphs and
categories~\cite[\S II.7]{MacLaneS:catwm}
and the adjunction between polynomial endofunctors and polynomial monads~\cite[\S4.6]{GambinoN:polfpm}. Hence, we both strengthen these adjointness results and put them in a
general context. Indeed, one of the original motivations for the research presented here was to make precise the
analogy between the two constructions. In both cases, the application of our main theorem simplifies
a problem regarding double categories by reducing it to a question on 2-categories. Note, however, that the combination of
horizontal and vertical arrows is exploited essentially to recover the existing results, since the free monad
construction acts on endomorphisms (which are defined using horizontal arrows) but its universal property is
expressed with respect to vertical endomorphism maps.

Some double-categorical aspects of monads have also been investigated  within the theory of
$\mathrm{fc}$-multicategories in~\cite[Chapter~5]{LeinsterT:higohc}
and within the theory of framed bicategories
in~\cite[\S11]{ShulmanM:frabmf}. However, the notion of a horizontal monad map
considered there generalizes the ring-theoretic notion of a
bimodule, whereas our horizontal monad maps are essentially
 the monad maps of Street~\cite{StreetR:fortm}.

\subsection*{Plan of the paper.} Section~\ref{sec:mon2cat} discusses monads in a 2-category, recalling
some basic notions from~\cite{StreetR:fortm} and giving a characterization of the free monads in a 2-category.
Section~\ref{sec:mondc} introduces the double category~$\Mnd(\bbC)$ of monads in a double category~$\bbC$
and illustrates its definition with examples.
Section~\ref{sec:monfb} establishes some special properties of $\Mnd(\bbC)$ under the assumption that $\bbC$
is a framed bicategory. In particular, we state our main result, Theorem~\ref{thm:freemndfrmbicat}, and
apply it to our examples. Finally, Section~\ref{sec:technical} contains the proof of Theorem~\ref{thm:freemndfrmbicat}.

\section{Monads in a 2-category}
\label{sec:mon2cat}

\subsection*{Preliminaries.} We recall some definitions concerning endomorphisms, monads and their algebras in a
2-category. Let $\catK$ be a 2-category.  An
\myemph{endomorphism} in $\catK$ is a pair~$(X,P)$ consisting of
an object~$X$ and a 1-cell $P : X \rightarrow X$. An \myemph{endomorphism map}~$(F,\phi) : (X,P) \rightarrow
(Y,Q)$ consists of a 1-cell $F : X \rightarrow Y$ and a 2-cell~$\phi : QF \rightarrow FP$, which is not
required to satisfy any condition. An \myemph{endomorphism
2-cell}~$\alpha : (F, \phi) \rightarrow (F', \phi')$ is a 2-cell~$\alpha : F \rightarrow F'$ making the following
diagram commute
\[
\xymatrix{
QF \ar[r]^{\phi} \ar[d]_{Q \alpha} & FP \ar[d]^{\alpha P} \\
QF' \ar[r]_{\phi'} & F'P .}
\]
We write~$\End(\catK)$ or $\End_\catK$ for the 2-category of endomorphisms, endomorphism maps and
endomorphism 2-cells in~$\catK$.  There is a 2-functor $\Inc : \catK \rightarrow \End(\catK)$ which sends an object $X \in \catK$
to the identity endomorphism $(X, 1_X)$ on $X$. Let us now consider a fixed endomorphism $(Y,Q)$ in $\catK$.
For $X \in \catK$, the category~$\Qalg_X$ of~$X$-indexed~$Q$-algebras, in the sense of Lambek,  is defined  by letting
\[
\Qalg_X \defeq \End_\catK( (X, 1_X), (Y,Q)) \, .
\]
Explicitly, an~$X$-indexed
$Q$-algebra consists of a 1-cell $F : X \rightarrow Y$, called the underlying 1-cell of the algebra,
and a 2-cell $f : QF \rightarrow F$, called the structure map of the algebra. Note that the structure
map is not required to satisfy any conditions. These definitions extend
to a 2-functor
\[
\Qalg_{(-)} : \catK \rightarrow \Cat \, .
\]
We write $U_{(-)} : \Qalg_{(-)} \rightarrow \catK(-,Y)$ for the 2-natural transformations whose components
are the forgetful functors $U_X : \Qalg_X \rightarrow \catK(X,Y)$ mapping an $X$-indexed $Q$-algebra to
its underlying $1$-cell.

\medskip

We write~$\Mnd(\catK)$ or $\Mnd_\catK$ for the 2-category of monads, monad maps and monad~2-cells in~$\catK$, as
defined in~\cite{StreetR:fortm}. As usual, we refer to a monad by mentioning only its underlying endomorphism,
leaving implicit its multiplication and unit.  With a minor abuse of notation,
we write $\Inc : \catK \rightarrow \Mnd(\catK)$ for the 2-functor mapping an object~$X$ to
the monad~$(X, 1_X)$. If~$(Y,Q)$ is a monad, for every $X \in \catK$ we may consider
not only the category~$\Qalg_X$ of Lambek algebras for its underlying endomorphism, but also the category
$\QAlg_X$ of $X$-indexed Eilenberg-Moore $Q$-algebras, which is defined by letting
\[
\QAlg_X \defeq \Mnd_\catK((X, 1_X), (Y,Q)) \, .
\]
Note that we write~$\Qalg_X$ for the category of algebras for the endomorphism and~$\QAlg_X$ for the
category of Eilenberg-Moore algebras for the monad. Explicitly, an~$X$-indexed Eilenberg-Moore $Q$-algebra
consists of a 1-cell~$F : X \rightarrow Y$ and a~2-cell~$f : QF \rightarrow F$ satisfying the axioms
\[
\xymatrix{
QQF \ar[r]^{Q f} \ar[d]_{\mu_Q F}  & QF \ar[d]^{f} \\
QF \ar[r]_{f} & F , } \qquad
\xymatrix{
F \ar[r]^{\eta_Q F} \ar@/_1pc/[dr]_{1_F}  & QF \ar[d]^{f}  \\
 & F .}
\]
Again, these definitions extend to a 2-functor $\QAlg_{(-)} : \catK \rightarrow \Cat$ and there is a~2-natural transformation $U_{(-)} : \QAlg_{(-)} \rightarrow \catK(-,Y)$, with components given by the
evident forgetful functors. Since $(Y,Q)$ is assumed to be a monad, for every $X \in \catK$ the forgetful functor
$U_X : \QAlg_X \rightarrow \catK(X,Y)$ has a left adjoint, defined by composition with $Q : Y \rightarrow Y$.

\subsection*{A characterization of free monads.}
We  generalize the characterization of the free monad on an endomorphism given by Staton
in~\cite[Theorem~6.1.5]{StatonS:namppc} from the 2-category~$\Cat$ to an arbitrary 2-category $\catK$. The
generalization is essentially straightforward, but we indicate the main steps of the proof. See~\cite[\S9.4]{BarrM:toptt}
for background material on free monads and~\cite{KellyGM:unittc} for a
general account of several examples of the free monad construction.

\begin{thm} \label{thm:frem2cat}
Let~$(Y,Q)$ be an endomorphism in a 2-category $\catK$. For a monad~$(Y,Q^*)$ and a 2-cell~$\iota_Q : Q \rightarrow Q^*$,
the following conditions are equivalent.
\begin{enumerate}[(i)]
\item The endomorphism map $(1_Y, \iota_Q) : (Y, Q^*) \rightarrow (Y,Q)$ is universal, in the sense that for every
monad~$(X,P)$, composition with $(1_Y, \iota_Q)$ induces an isomorphism fitting in the diagram
 \begin{equation}
 \label{equ:iso2cat}
 \xycenter{
\Mnd_\catK( (X,P), (Y, Q^*) ) \ar[rr]^{\iso} \ar[dr] & & \End_\catK( (X,P), (Y,Q)) \ar[dl] \\
 & \catK(X,Y), & }
 \end{equation}
where the downward arrows are the evident forgetful functors.
\item The 2-cell $\nu_{Q^*} : Q Q^* \rightarrow Q^*$, defined as the composite
\begin{equation}
\label{equ:nu}
\xymatrix@C=8ex{ Q Q^*  \ar[r]^{\iota_{Q} Q^* } & Q^* Q^*  \ar[r]^{\mu_{Q^*} } & Q^* , }
\end{equation}
equips $Q^*$ with a universal $Q$-algebra structure, in the sense that
for every $X \in \catK$, the functor $\catK(X,Y) \rightarrow \Qalg_X$ defined by mapping
$F : X \rightarrow Y$ to the $Q$-algebra with underlying 1-cell $Q^* F$ and structure map the 2-cell
$\nu_{Q^*} F : Q Q^* F \rightarrow Q^* F$, is left adjoint to the forgetful functor $U_X : \Qalg_X \rightarrow \catK(X,Y)$.
\end{enumerate}
\end{thm}

\begin{proof} To see that (i) implies (ii), consider the following diagram:
\[
\xymatrix{
\QQalg_X \ar[rr]^{\iso} \ar@{=}[d] & &  \Qalg_X \ar@{=}[d] \\
\Mnd_\catK( (X,1_X), (Y, Q^*) )  \ar[dr] \ar[rr]^{\iso} & & \End_\catK( (X,1_X), (Y,Q)) \ar[dl] \\
 & \catK(X, Y), }
 \]
 where the bottom triangular diagram is an instance of the diagram in~\eqref{equ:iso2cat}.
The functor defined in~(ii) is left adjoint to the forgetful functor
$U_X : \Qalg_X \rightarrow \catK(X,Y)$ since it is exactly the composite of the left adjoint
$\catK(X,Y) \rightarrow \QQalg_X$, which is given by composition with~$Q^*$ (since $Q^*$ is a monad), with the
isomorphism $\QQalg_X \rightarrow \Qalg_X$, which is defined by composition with~$\iota_Q$.

\smallskip

For the proof that~(ii) implies~(i), we need to define an isomorphism as in~\eqref{equ:iso2cat}.
Given an endomorphism map $(F, \phi) : (X,P) \rightarrow (Y,Q)$,
where~$\phi : QF \rightarrow FP$, we need to define a monad map $(F, \phi^\sharp) : (X,P) \rightarrow (Y,Q^*)$,
where~$\phi^\sharp : Q^*F \rightarrow FP$. For this, we exploit the adjointness in~(ii). Note that the
left adjoint to $\Qalg_X \rightarrow \catK(X,Y)$ sends~$F$ to the $Q$-algebra with underlying 1-cell $Q^*F$
and structure map $\nu_{Q^*} F : Q Q^* F \rightarrow Q^* F$. Now, observe that the map
\[
\xymatrix{ QFP \ar[r]^{\phi P} & FPP \ar[r]^{F \mu_P} & FP }
\]
equips~$FP$ with a $Q$-algebra structure. By adjointness, the map $\phi^\sharp : Q^* F \rightarrow FP$
is defined as the unique $Q$-algebra morphism such that the following diagram commutes
\[
\xymatrix{
F \ar[r]^-{\eta_{Q^*} F} \ar@/_1pc/[dr]_{F \eta_P} & Q^* F \ar[d]^{\phi^\sharp} \\
 & FP .}
 \]
Note that saying that $\phi^\sharp$ is a $Q$-algebra morphism amounts to saying that the following diagram commutes
\[
\xymatrix{
Q Q^* F \ar[r]^{Q \phi^\sharp} \ar[dd]_{\nu_{Q^*} F} & Q FP \ar[d]^{\phi P} \\
  & F P P \ar[d]^{F \mu_P} \\
Q^* F \ar[r]_{\phi^\sharp} & F P .}
\]
The isomorphism is defined as the identity on 2-cells. It remains to check that what we have defined
is indeed an inverse to the functor defined by composition with $(1_Y, \iota_Q)$, but the verification
is essentially identical to the one given in detail in the proof of~\cite[Theorem~6.1.5]{StatonS:namppc}
and hence we omit it.
\end{proof}

\begin{defn} A 2-category $\catK$ is said to \myemph{admit the construction of free monads} if for every
endomorphism $(Y,Q)$ there exists a monad $(Y,Q^*)$ and a 2-cell $\iota_Q : Q \rightarrow Q^*$ satisfying the
equivalent conditions of Theorem~\ref{thm:frem2cat}.
\end{defn}

\begin{rmk} \label{thm:usefulrmk}
Let us point out that the universal property of the free monad $(Y, Q^*)$
on an endomorphism $(Y,Q)$ stated in item (i) of Theorem~\ref{thm:frem2cat} includes the assertion that
for every monad~$(X,P)$ and every endomorphism map~$(F, \phi) : (X, P) \rightarrow (Y,Q)$,
there exists a unique 2-cell~$\phi^\sharp : Q^* F \rightarrow FP$
such that $(F, \phi^\sharp) : (X, P) \rightarrow (Y,Q^*)$ is a monad map and the diagram
\[
\xymatrix{
QF \ar[r]^{\iota_Q F}  \ar@/_1pc/[dr]_{\phi} & Q^* F \ar[d]^{\phi^\sharp} \\
 & FP }
\]
commutes. From the statement in item (ii) of Theorem~\ref{thm:frem2cat},
it also follows that if $\catK$ is a 2-category that admits the
construction of free monads and has local coproducts,
\myemph{i.e.}~coproducts in its hom-categories,  then for every $F : X \rightarrow Y$,
the initial algebra for the endofunctor
\[
 \begin{array}{ccc}
  \catK(X,Y) & \rightarrow & \catK(X,Y) \\
  (-) & \mapsto & F + Q \, (-)
\end{array}
  \]
has $Q^* F$ as its underlying object and the copair of the
2-cells~$\eta_{Q^*} F : F  \rightarrow Q^* F$ and~$\nu_{Q^*} F : Q Q^* F \rightarrow Q^* F$ as
its structure map.
\end{rmk}

By the bicategorical Yoneda lemma~\cite{StreetR:fibb}, every bicategory is biequivalent to
a 2-category~\cite[Theorem~1.4]{GordonR:coht}. Hence, the remarks and the results above
can be applied also to bicategories. We now introduce our two main classes of examples:
bicategories of spans and bicategories of polynomials.

\begin{examp} \label{thm:span2cat}
Let $\catE$ be a category with finite limits. Recall that a span in $\catE$ is a diagram of the form
\begin{equation}
\label{equ:span}
\xycenter{  & F \ar[dr]^{\tau} \ar[dl]_{\sigma} &  \\
X & &  Y , }
\end{equation}
and that a span morphism is a commutative diagram of the form
\begin{equation}
\label{equ:spanmorphism2cat}
\xycenter{
X  \ar@{=}[d] & F \ar[r]^{\tau} \ar[l]_{\sigma} \ar[d]^{\phi} & Y \ar@{=}[d] \\
X & F' \ar[r]_{\tau'} \ar[l]^{\sigma'} & Y .}
\end{equation}
We write $\Sp_\catE$ for the bicategory of spans in $\catE$, originally defined in~\cite{BenabouJ:intb}, which has the
objects of $\catE$ as 0-cells, spans  as 1-cells  and span morphisms as 2-cells. It is well-known that graphs and
categories in $\catE$ can be identified with endomorphisms and monads in  $\Sp_\catE$~\cite{BenabouJ:intb,BurroniA:tcat}.
For our purposes, it is convenient to recall the definition of the 2-category of linear functors over $\catE$, which
is biequivalent to the bicategory $\Sp_\catE$. Given a span as in~\eqref{equ:span}, we define its associated
linear functor to be the composite
\[
\xymatrix{ \catE/X \ar[r]^{\pbk{\sigma}} & \catE/F \ar[r]^{\ladj{\tau}} & \catE/Y  } ,
\]
where~$\pbk{\sigma}$ acts by pullback along~$\sigma$ and~$\ladj{\tau}$ acts by composition with~$\tau$. In
general, a functor between slices of $\catE$ is said to be linear if it is naturally isomorphic to a functor
of this form. Now, recall from~\cite[\S1.3]{GambinoN:polfpm} that slice categories of $\catE$ are tensored
over~$\catE$ and that linear functors have a canonical strength. The 2-category of linear functors is then defined as
the sub-2-category of~$\Cat$ having slice categories of~$\catE$ as 0-cells, linear functors between them as 1-cells,
and strong natural transformations as 2-cells, \myemph{i.e.} natural transformations compatible with the
canonical strength on linear functors. Let us also recall that a strong natural transformation between linear functors
is cartesian, \myemph{i.e.}~its naturality squares are pullbacks. By the biequivalence,
graphs in~$\catE$ can be thought of as linear endofunctors and categories in $\catE$ can be thought of as linear monads,
\textit{i.e.}~monads whose underlying functor is linear and whose multiplication and unit are strong natural
transformations.
\end{examp}

\begin{examp} \label{thm:poly2cat}
Let~$\catE$ be a locally cartesian closed category. Recall from~\cite[\S1.4]{GambinoN:polfpm} that a
polynomial over $\catE$ is a diagram of the form
\begin{equation}
\label{equ:poly}
\xycenter{
 & \bar{F} \ar[dl]_{\sigma} \ar[r]^{\theta} & F \ar[dr]^{\tau} & \\
 X & & & Y }
 \end{equation}
and a cartesian morphism of polynomials is a diagram  of the form
\[
\xycenter{
X  \ar@{=}[d] & \bar{F} \ar[d] \ar[r] \ar[l] \drpullback  & F \ar@{=}[d]  \ar[r] & Y \ar@{=}[d] \\
X  & \bar{F}'  \ar[r] \ar[l] & F' \ar[r] & Y , }
\]
where the central square is a pullback. We write $\Poly_\catE$ for the bicategory of polynomials over $\catE$, as defined
in~\cite[\S1.16]{GambinoN:polfpm}, which has the objects of $\catE$ as 0-cells, polynomials as 1-cells, and
cartesian morphisms of polynomials as 2-cells. Working in the internal logic of $\catE$, for a polynomial as
in~\eqref{equ:poly} we may represent an element $f \in F$ as an arrow
\[
f : ( x_i \ | \ i \in I   ) \rightarrow y \, ,
\]
where $I \defeq \theta^{-1}(f)$,
the family $( x_i \ | \ i \in I   )$ is defined by letting $x_i \defeq \sigma(i)  $, for~$i~\in~I$, and
$y \defeq  \tau(f)$. Thus, we think of the set $I$ as the arity of the arrow~$f$.
The biequivalence between the bicategory of spans and the 2-category of linear functors
extends to a biequivalence between the bicategory of polynomials and the 2-category of polynomial
functors~\cite[Theorem~2.17]{GambinoN:polfpm}, as we now proceed to recall. For a
polynomial as in~\eqref{equ:poly}, the polynomial functor associated to it is defined as the composite
\[
\xymatrix{
\catE/X \ar[r]^{\pbk{\sigma}} & \catE/ \bar{F} \ar[r]^{\radj{\theta}} & \catE/F \ar[r]^{\ladj{\tau}} &
\catE/Y },
\]
where~$\radj{\theta}$ is the right adjoint to the pullback functor~$\pbk{\theta}$. A functor between
slices of $\catE$ is said to be polynomial if it is naturally isomorphic to a functor of this form.
Like linear functors, polynomial functors have a canonical strength and so we can define the 2-category
of polynomial functors as
the sub-2-category of $\Cat$ having slices of $\catE$ as~0-cells, polynomial functors as~1-cells and cartesian
strong natural transformations as~2-cells. The biequivalence between $\Poly_\catE$ and the 2-category of
polynomial functors allows us to identify endomorphisms and monads in~$\Poly_\catE$ with polynomial
endofunctors and polynomial monads on slices of $\catE$, respectively, where by a polynomial monad we mean a monad whose
underlying endofunctor is polynomial and whose multiplication and unit are cartesian strong natural transformations.

Let us also recall from~\cite{MoerdijkI:weltc} that a locally cartesian closed category $\catE$ is said to have W-types
if every polynomial endofunctor $P : \catE \rightarrow \catE$ has an initial algebra, called the W-type of the functor.
Note that a polynomial functor $P : \catE \rightarrow \catE$ has to be represented by a diagram as in~\eqref{equ:poly}
in which both $X$ and $Y$ are the terminal object of $\catE$ and hence is competely determined by the map $\theta$.
The category-theoretic notion of a W-type is a counterpart of the notion of a type of wellfounded trees, originally
introduced by Martin-L\"of within his dependent type theory~\cite{NordstronB:promlt}. As shown
in~\cite[Theorem~12]{GambinoN:weltpf}, if $\catE$ has disjoint coproducts, the assumption of W-types is sufficient
to show that, for all $X \in \catE$,  every polynomial endofunctor $P : \catE/X \rightarrow \catE/X$ has an initial
algebra. For further material and references on polynomial functors, see~\cite{GambinoN:polfpm} and its bibliography.
\end{examp}

Proposition~\ref{thm:key2cat} provides the horizontal part of Proposition~\ref{thm:freemndexamp}. Item~(i)
in its statement refers to the notion of a pretopos with parametrized list objects, for which we invite the reader
to refer to~\cite{MaiettiM:joyaul}.

\begin{prop} \label{thm:key2cat} \, \hfill
\begin{enumerate}[(i)]
\item If $\catE$ is a pretopos with parametrized list objects,
the bicategory  $\Sp_\catE$ admits the construction of free monads.
\item If $\catE$ is a  locally cartesian closed category  with disjoint coproducts and
W-types, the bicategory $\Poly_\catE$ admits the construction of free monads.
\end{enumerate}
\end{prop}

\begin{proof} We begin by proving (ii). We exploit the biequivalence between the bicategory of polynomials
and the
2-category of polynomial functors. Let $Q : \catE/Y \rightarrow \catE/Y$ be a polynomial endofunctor.
We show that there is a polynomial monad $Q^* : \catE/Y \rightarrow \catE/Y$ and a cartesian strong
natural transformation $\iota  : Q \rightarrow Q^*$ that satisfy the universal property in item (ii)
of Theorem~\ref{thm:frem2cat}. By~\cite[Theorem~12]{GambinoN:weltpf}, the assumption that $\catE$ has W-types
implies that the forgetful
functor $U : \Qalg \rightarrow \catE/Y$ has a left adjoint. We let~$Q^* : \catE/Y \rightarrow \catE/Y$ be the
monad resulting from the adjunction. The monad $Q^* : \catE/Y \rightarrow \catE/Y$ is polynomial
by~\cite[Theorem~4.5]{GambinoN:polfpm}. If~$Q : \catE/Y \rightarrow \catE/Y$ is represented by the polynomial
\begin{equation}
\label{equ:endoq}
\xycenter{
  & \bar{Q} \ar[dl]  \ar[r]^{\theta_Q}  &  Q  \ar[dr]  & \\
Y &  & & Y }
\end{equation}
then $Q^* : \catE/Y \rightarrow \catE/Y$ is represented by the polynomial
\begin{equation}
\label{equ:mndq}
\xycenter{
  & \bar{Q}^* \ar[dl]  \ar[r]^{\theta_{Q^*}} &  Q^*  \ar[dr]   & \\
Y & & & Y ,}
\end{equation}
where the object~$Q^*$ in~\eqref{equ:mndq} is described in the internal logic of~$\catE$ as the set
of wellfounded trees of profile $Q$, \myemph{i.e.}~trees built up from identities and formal composites
of the arrows in $Q$. The map $\theta_{Q^*}$ in~\eqref{equ:mndq} describes the arities of the arrows
in $Q^*$ in the evident way. The inclusion of the arrows in $Q$ into
those in $Q^*$ is part of a diagram
\begin{equation}
\label{equ:iotacart}
\xycenter{
Y \ar@{=}[d] &  \bar{Q} \ar[r]^{\theta_Q}  \ar[l]  \ar[d] \drpullback & Q  \ar[d] \ar[r]  & Y \ar@{=}[d] \\
Y  & \bar{Q}^* \ar[r]_{\theta_{Q^*}}  \ar[l]  & Q^* \ar[r]   & Y ,}
\end{equation}
which represents the required cartesian strong natural transformation~$\iota : Q \rightarrow Q^*$.
A direct verification shows that the left adjoint to~$U : \Qalg \rightarrow \catE/Y$ maps an object~$A$
to the~$Q$-algebra with underlying object~$Q^*A$ and structure map~$\nu_A : QQ^*A \rightarrow Q^*A$, where
$\nu_{Q^*} : QQ^* \rightarrow Q^*$ is defined as in~\eqref{equ:nu}.
To conclude the proof of item (ii) it is sufficient to observe that, for $X \in \catE$, the category $\Qalg_X$
is equivalent to the category of polynomial functors $F : \catE/X \rightarrow \catE/Y$ equipped with
a  cartesian strong natural transformation $\phi : QF \rightarrow F$.

\smallskip

The proof of item~(i) is similar, except that polynomial functors are replaced by linear
functors. In this case, the assumption of W-types can be replaced by that of parametrized
list objects, which suffice to prove the existence of the left adjoint to the forgetful
functor~$U : \Qalg \rightarrow \catE/Y$ and that the resulting monad $Q^* : \catE/Y \rightarrow \catE/Y$
is linear. This is because linear endofunctors (respectively, linear monads) are just graphs (respectively, categories)
internal to $\catE$, and, as shown in~\cite[Proposition 7.3]{MaiettiM:joyaul}, the assumption of
parametrized list objects guarantees the existence of the free category on a graph in~$\catE$.
\end{proof}

If~$\catE$ is a locally cartesian closed pretopos with W-types, then it has list objects and these
are parametrized since we are in a cartesian closed category. Hence, such a category satisfies the hypotheses
of both item~(i) and item~(ii) of Proposition~\ref{thm:key2cat}. In this case, the construction of the free monad
for polynomial endofunctors generalizes the construction of the free monad for linear endofunctors.

 \section{Monads in a double category}
\label{sec:mondc}

\subsection*{Notation and preliminaries.}

We assume  readers to be familiar with the basic concepts of the
theory of double categories (see~\cite{EhresmannC:cats} for the
original reference and~\cite{FioreT:modscs,GrandisM:limdc,GrandisM:adjdc}  for
modern accounts) and limit ourselves to introducing some notation
and recalling some basic notions. For a double
category~$\bbC$, we write~$\Obj_\bbC$ for its class of objects,
$\Hor_\bbC$ for its class of horizontal arrows, $\Ver_\bbC$ for its
class of vertical arrows and~$\Sq_\bbC$ for its class of squares.
We write~$\bbC_0$ for the category of objects and vertical arrows
and~$\bbC_1$ for the category of horizontal arrows and squares.
We allow horizontal composition to be
associative and unital up to coherent invertible squares rather
than strictly. For the sake of readability, however, we shall
work as if horizontal composition were strict, as allowed
by~\cite[Theorem~7.5]{GrandisM:limdc}. Typically,  a square will be written
as follows:
\begin{equation}
\label{equ:typicalsquare}
\xycenter{ X  \ar[r]^F \ar[d]_{u}  \ar@{}[dr]|{\alpha} & Y \ar[d]^{v} \\
X' \ar[r]_{F'} &  Y'. }
\end{equation}
Identity squares will be written without a label, as follows:
\[
\xymatrix{
X \ar[r]^F \ar@{=}[d]  & Y \ar@{=}[d]  \\
X \ar[r]_F & Y,} \qquad
\xymatrix{
X \ar@{=}[r] \ar[d]_u & X \ar[d]^u \\
X' \ar@{=}[r] & X'. }
\]
For a double category $\bbC$, its horizontal 2-category $\mathcal{H}_\bbC$ is defined as follows: the 0-cells
are the objects of $\bbC$, the 1-cells are the horizontal arrows of $\bbC$ and the 2-cells are the squares of
the form
\[
\xymatrix{
X \ar[r]^{F} \ar@{=}[d] \ar@{}[dr]|{\alpha} & Y \ar@{=}[d] \\
X \ar[r]_{F'} & Y .}
\]
The notions of horizontal adjunction and vertical adjunction between double categories can be defined using
the general notion of an adjunction in a 2-category~\cite{KellyGM:reve2c}. A horizontal adjunction
is an adjunction in the 2-category of double categories, double functors and horizontal natural
transformations; vertical adjunctions are defined analogously, replacing horizontal natural transformations
with vertical ones~\cite{GrandisM:limdc}.

\begin{examp} \label{thm:span} Let $\catE$ be a category with finite limits. With a minor abuse
of notation, we write $\Sp_\catE$ also for the double category of spans in $\catE$, which has
objects of $\catE$ as objects, spans as horizontal arrows, maps of $\catE$ as vertical arrows
and diagrams of the form
\[
\xymatrix{
X \ar[d]_u & F \ar[r]^{\tau} \ar[d]^{\phi} \ar[l]_{\sigma} & Y \ar[d]^v \\
X' & F' \ar[l]^{\sigma'} \ar[r]_{\tau'} & Y' }
\]
as squares. Note that the horizontal bicategory of this double category is exactly the bicategory
of spans in $\catE$ defined in Example~\ref{thm:span2cat}.
\end{examp}

\begin{examp} Let $\catE$ be a locally cartesian closed category. With another
abuse of notation, we write $\Poly_\catE$ also for the double category of polynomials over $\catE$,
which has the objects of $\catE$ as objects, polynomials as horizontal arrows, maps of $\catE$ as
vertical arrows and diagrams of the form
\[
\xymatrix{
X \ar[d]_u & \bar{F} \ar[l]_{\sigma} \drpullback \ar[r]^{\theta} \ar[d] & F \ar[r]^{\tau} \ar[d]^{\phi}  & Y \ar[d]^v \\
X' &  \bar{F'}\ar[l]^{\sigma'} \ar[r]_{\theta'} & F'  \ar[r]_{\tau'} & Y' ,}
\]
where the central square is a pullback, as squares. The bicategory of polynomials defined in Example~\ref{thm:poly2cat}
is the horizontal bicategory of this double category.
 \label{thm:poly}
\end{examp}

\subsection*{The double categories of endomorphisms and monads.} Below, we define the double category $\Mnd(\bbC)$ of monads in a double
 category~$\bbC$. After giving the definition, we  explain how it generalizes the definition of the 2-category $\Mnd(\catK)$
of monads in a 2-category $\catK$.
In view of our applications, we begin by introducing the double category $\End(\bbC)$ of
endomorphisms in a double category~$\bbC$.

\begin{defn} \label{thm:end} Let $\bbC$ be a double category.
\begin{enumerate}[(i)]
\item A \myemph{horizontal  endomorphism} is a pair  $(X,P)$ consisting of an object~$X$
and a horizontal arrow $P : X \rightarrow X$. Since we consider only horizontal endomorphisms,
we refer to them simply as endomorphisms.
\item A \myemph{horizontal endomorphism map} $(F, \phi) : (X,P) \rightarrow (Y,Q)$
 consists of a horizontal arrow $F \colon X \rightarrow Y$ and a square
\[
\xymatrix{
X \ar[r]^{F} \ar@{=}[d]   \ar@{}[drr]|{\phi}   & Y \ar[r]^{Q} & Y \ar@{=}[d]  \\
X \ar[r]_{P} & X \ar[r]_{F} & Y .}
\]
\item A \myemph{vertical endomorphism map} $(u, \bar{u}) : (X,P) \rightarrow (X',P')$
consists of a vertical arrow $u : X \rightarrow X'$ and a square
\[
\xymatrix{
X \ar[r]^{P} \ar[d]_{u} \ar@{}[dr]|{\bar{u}}   & X \ar[d]^{u} \\
X'  \ar[r]_{P'} & X' .}
\]
\item An \myemph{endomorphism square}
\[
 \xymatrix@C=8ex {
 (X,P) \ar[r]^{(F, \phi)}   \ar[d]_{(u, \bar{u})}   \ar@{}[dr]|{\alpha} & (Y,Q) \ar[d]^{(v, \bar{v})} \\
 (X',P')  \ar[r]_{(F', \phi')} & (Y',Q') }
\]
is a square
 \[
 \xymatrix {
 X  \ar[r]^{F}   \ar[d]_{u}   \ar@{}[dr]|{\alpha} & Y \ar[d]^{v}  \\
 X'  \ar[r]_{F'} & Y' }
\]
satisfying the condition
\[
\xycenter{
X  \ar[r] \ar@{=}[d] \ar@{}[drr]|{\phi}  & Y  \ar[r] & Y \ar@{=}[d] \\
X  \ar[r] \ar[d]  \ar@{}[dr]|{\bar{u}}  & X  \ar[r] \ar[d]  \ar@{}[dr]|{\alpha}  & Y \ar[d]  \\
X' \ar[r]  & X' \ar[r]  & Y'} \quad = \quad \xycenter{
X \ar[r] \ar[d]  \ar@{}[dr]|{\alpha} &  Y   \ar[r]  \ar[d]  \ar@{}[dr]|{\bar{v}} & Y \ar[d]  \\
X'  \ar[r]   \ar@{=}[d]  \ar@{}[drr]|{\phi'}  & Y'  \ar[r]  & Y' \ar@{=}[d] \\
X'  \ar[r]  & X'  \ar[r]   & Y' .}
\]
\end{enumerate}
\end{defn}

We write $\End(\bbC)$ for the double category of  endomorphisms,
horizontal endomorphism maps, vertical endomorphism maps and
endomorphism squares. We omit the straightforward verification
that~$\End(\bbC)$ is indeed a double category.

\begin{defn} \label{thm:mnd} Let $\bbC$ be a double category.
\begin{enumerate}[(i)]
\item A \myemph{monad} is an endomorphism $(X,P)$ equipped with squares
\[
\xymatrix{
X \ar[r]^{P} \ar@{=}[d]  \ar@{}[drr]|{\mu_P} & X \ar[r]^{P} & X \ar@{=}[d] \\
X \ar[rr]_{P} & & X } \qquad \xymatrix{
X \ar@{=}[r]   \ar@{=}[d]  \ar@{}[dr]|{\eta_P}  & X \ar@{=}[d] \\
X \ar[r]_{P} & X }
\]
satisfying the associativity law
\[
\xycenter{
X \ar[r]  \ar@{=}[d]  \ar@{}[drr]|{\mu_P} & X \ar[r]  & X \ar[r]  \ar@{=}[d]  & X \ar@{=}[d] \\
X \ar[rr]  \ar@{=}[d]  \ar@{}[drrr]|{\mu_P}  & & X \ar[r] & X  \ar@{=}[d]  \\
X \ar[rrr]  & & & X }  \quad =  \quad
\xycenter{
X \ar[r] \ar@{=}[d]    & X \ar[r] \ar@{=}[d]  \ar@{}[drr]|{\mu_P} & X \ar[r] & X \ar@{=}[d] \\
X \ar[r] \ar@{=}[d] \ar@{}[drrr]|{\mu_P} & X \ar[rr] & & X \ar@{=}[d] \\
X \ar[rrr] & & & X }
\]
and the unit laws
\[
\xycenter{
X \ar[r] \ar@{=}[d]  & X \ar@{=}[d] \ar@{=}[r]  \ar@{}[dr]|{\eta_P} & X \ar@{=}[d] \\
X \ar[r] \ar@{=}[d] \ar@{}[drr]|{\mu_P}  & X \ar[r] & X \ar@{=}[d] \\
X \ar[rr] & & X } \quad = \quad
\xycenter{
X \ar[r] \ar@{=}[d]   & X \ar@{=}[d] \\
X \ar[r] & X }  \quad = \quad
\xycenter{
X \ar@{=}[r]   \ar@{=}[d] \ar@{}[dr]|{\eta_P}  & X \ar@{=}[d] \ar[r]   & X \ar@{=}[d] \\
X \ar[r] \ar@{=}[d]  \ar@{}[drr]|{\mu_P} & X \ar[r] & X \ar@{=}[d] \\
X \ar[rr] & & X. }
\]
As before, we refer to a monad as above by mentioning only its underlying endomorphism~$(X,P)$.
\item A \myemph{horizontal monad map} $(F, \phi) : (X,P) \rightarrow (Y,Q)$
is a horizontal endomorphism map between the underlying endomorphisms
 satisfying the following conditions:
\[
\xycenter{
X \ar@{=}[d] \ar[r]  & Y \ar@{=}[d] \ar[r]  \ar@{}[drr]|{\mu_Q} & Y \ar[r]  & Y \ar@{=}[d] \\
X \ar@{=}[d] \ar[r]  \ar@{}[drrr]|{\phi } & Y \ar[rr] &      & Y \ar@{=}[d] \\
X \ar[rr]  & & X \ar[r]  & Y }  \quad = \quad \xycenter{
X  \ar@{=}[d]  \ar[r] \ar@{}[drr]|{\phi } & Y  \ar[r] & Y  \ar[r] \ar@{=}[d]   & Y \ar@{=}[d]  \\
X  \ar@{=}[d] \ar[r]   & X  \ar[r] \ar@{=}[d]  \ar@{}[drr]|{\phi } & Y  \ar[r] & Y  \ar@{=}[d] \\
X  \ar@{=}[d]  \ar[r]  \ar@{}[drr]|{\mu_P}  & X  \ar[r] & X  \ar[r] \ar@{=}[d]  & Y \ar@{=}[d]  \\
X   \ar[rr] & & X  \ar[r] & Y   }
\]

\[
\xycenter{
X \ar@{=}[d]  \ar[r]   & Y \ar@{=}[r] \ar@{=}[d]   \ar@{}[dr]|{\eta_Q} & Y \ar@{=}[d]  \\
X \ar@{=}[d] \ar[r] \ar@{}[drr]|{\phi }  & Y \ar[r]  & Y \ar@{=}[d]  \\
X \ar[r]  & X \ar[r]  & Y } \quad = \quad \xycenter{
X \ar@{=}[d] \ar@{=}[r]   \ar@{}[dr]|{\eta_P}   & X \ar[r]  \ar@{=}[d]   & Y \ar@{=}[d]  \\
X \ar[r]  & X \ar[r]  & Y.}
\]
\item A \myemph{vertical monad map} $(u, \bar{u}) : (X,P) \rightarrow (X',P')$
is a vertical endomorphism map between the underlying endomorphisms
satisfying the following conditions:
\[
\xycenter{
X \ar[r] \ar@{=}[d] \ar@{}[drr]|{\mu_P}  & X \ar[r] & X \ar@{=}[d] \\
X \ar[rr]  \ar[d]  \ar@{}[drr]|{\bar{u} } & & X \ar[d] \\
X' \ar[rr]  & & X' } \quad = \quad
\xycenter{
X \ar[r] \ar[d] \ar@{}[dr]|{\bar{u}} & X \ar[r]  \ar[d]  \ar@{}[dr]|{\bar{u}}  & X \ar[d]  \\
X' \ar[r]   \ar@{=}[d]  \ar@{}[drr]|{\mu_{P'}} & X' \ar[r]   & X' \ar@{=}[d] \\
X' \ar[rr]  & & X' }
\]

\[
\xycenter{
X \ar@{=}[r]  \ar@{=}[d]   \ar@{}[dr]|{\eta_P}  & X \ar@{=}[d] \\
X \ar[r] \ar[d] \ar@{}[dr]|{\bar{u}}  & X \ar[d] \\
X' \ar[r] & X'}  \quad = \quad
\xycenter{
X \ar@{=}[r]  \ar[d]   & X \ar[d]  \\
X' \ar@{=}[r]  \ar@{=}[d] \ar@{}[dr]|{\eta_{P'}} & X' \ar@{=}[d] \\
X' \ar[r]  & X'. }
\]
\item A \myemph{monad square} is an endomorphism square between the underlying
endomorphism maps.
\end{enumerate}
\end{defn}

We write~$\Mnd(\bbC)$ for the double category
of monads, horizontal monad maps, vertical monad maps and
monad squares; again, it is straightforward
to check that~$\Mnd(\bbC)$ is a double category. Before giving examples, we clarify the relationship between
our definitions and those in~\cite{StreetR:fortm}.

\begin{rmk} \label{thm:comparison}  Let~$\catK$ be a 2-category and
consider the double category~$\mathbb{H}(\catK)$, that has~$\catK$
as its horizontal 2-category and only identity 1-cells as vertical
arrows. Monads in~$\catK$ are the same as monads
in~$\mathbb{H}(\catK)$ and monad maps in~$\catK$ are the same
 as horizontal monad maps in~$\mathbb{H}(\catK)$. Finally,
monad 2-cells in~$\catK$ are the same as monad squares
in~$\mathbb{H}(\catK)$ of the special form
\[
\xymatrix@C=8ex{
(X,P) \ar[r]^{(F,\phi)}\ar@{=}[d] \ar@{}[dr]|{\alpha} & (Y,Q) \ar@{=}[d]  \\
(X,P) \ar[r]_{(F',\phi')} & (Y,Q). }
\]
In particular, the horizontal 2-category of~$\Mnd(
\mathbb{H}(\catK))$ is the 2-category~$\Mnd(\catK)$
of~\cite{StreetR:fortm}. As we explain in the following examples,
the presence of non-trivial vertical arrows in a double category
allows us to describe important mathematical structures as vertical
monad maps.
\end{rmk}

\begin{examp}
Let~$\catE$ be a category with finite limits. The category~$\Grph_\catE$ of graphs and graph morphisms internal
to~$\catE$ can be identified with the category of endomorphisms and vertical endomorphism maps
in the double category $\Sp_\catE$, while the category~$\Cat_\catE$ of categories and functors internal to~$\catE$
can be identified with the category of monads and vertical monad maps in~$\Sp_\catE$. We see
here an example of the benefits of considering monads in a double category rather than in a 2-category: while categories
can be seen as monads in the bicategory of spans in~$\catE$, functors between
categories are not the same as monad maps in that
bicategory.
\end{examp}

\begin{examp} Let $\catE$ be a locally cartesian closed category with finite disjoint coproducts and W-types.
We write $\PolyEnd_\catE$ for the category  of endomorphisms and vertical endomorphism maps
in the double category~$\Poly_\catE$ and write $\PolyMnd_\catE$ for the category of monads and vertical monad maps in $\Poly_\catE$.
If $M: \catE\to \catE$ is the free monoid monad in $\catE$
(which exists by the assumptions on $\catE$), then
  there is a double category $\PolyEnd_{\catE}/M$ whose objects are
  endomorphisms with a vertical endomorphism map to~$M$.  This is the
  double category of~$M$-spans in the sense of~\cite{BurroniA:tcat}
  and~\cite{LeinsterT:higohc}, while~$\PolyMnd_{\catE}$ is the double category of multicategories.
  The free monad on an endofunctor over~$M$ is the free multicategory
  on an $M$-span.  Furthermore, the vertical maps in
  $\PolyMnd_{\catE}$ are the multifunctors, and hence we see again
  the benefits of considering monads in the double categories rather
  than just in $2$-categories.  Further variations are possible: with
  a polynomial monad $T$ in the place of $M$ we get the same result
  for $T$-spans and $T$-multicategories, and in the particular case
  where $T$ is the identity monad, we are back to just plain
  categories in $\catE$.
\end{examp}

The function sending a monad $(X,P)$ to its underlying object $X$ extends
to a double functor $\Und : \Mnd(\bbC) \rightarrow \bbC$ and the function mapping an object $X \in \bbC$ to
the identity monad $(X, 1_X)$ extends to a double functor $\Inc : \bbC
\rightarrow \Mnd(\bbC)$. It is easy to check that $\Inc$ is a horizontal right adjoint
to~$\Und$, essentially as in the 2-categorical formal theory of monads~\cite[Theorem~1]{StreetR:fortm}.
The question of when~$\bbC$ admits the construction of Eilenberg-Moore objects, that is, of
when the double functor $\Inc$ has a horizontal right adjoint, will be treated in a sequel to this
paper. Here, instead, we focus on the construction of free monads.

\subsection*{Free monads in a double category.} We write $U : \Mnd(\bbC)
\rightarrow \End(\bbC)$ for the forgetful double functor mapping a monad to its underlying endomorphism.

\begin{defn} \label{thm:deffrem}
A double category $\bbC$ is said to \myemph{admit
the construction of free monads} if
$U : \Mnd(\bbC) \rightarrow \End(\bbC) $ has a vertical left adjoint.
\end{defn}

\begin{rmk} \label{thm:charadj} We now make explicit what it means
for a double category $\bbC$ to admit the construction of free monads.
By an analogue of  the characterization of ordinary adjunctions in
terms of universal arrows~\cite[Theorem IV.2]{MacLaneS:catwm}, to
give a vertical left adjoint to $U$ amounts to giving the following
data in~\eqref{item:adj1}-\eqref{item:adj4} satisfying the functoriality
condition in~$(\ast)$.
\begin{enumerate}[(i)]
\item \label{item:adj1} For every endomorphism $(X,P)$, a monad $(X^*,P^*)$.
\item \label{item:adj2} For every endomorphism $(X,P)$, a universal vertical
endomorphism map
\[
(\iota_X, \iota_P) : (X,P) \rightarrow (X^*, P^*) \, .
\]
Universality means that for
each vertical endomorphism map $(u, \bar{u}) \colon (X,P) \rightarrow (X',P')$, where $(X',P')$ is a monad, there exists a unique vertical
monad map $(u^\sharp, \bar{u}^\sharp) : (X^*,P^*) \rightarrow
(X',P')$ such that
\[
\vcenter{\hbox{ \xymatrix{
X \ar[r]^P \ar[d]_{u} \ar@{}[dr]|{\bar{u}} & X \ar[d]^{u} \\
X' \ar[r]_{P'} & X' }}} \quad =  \quad \vcenter{\hbox{ \xymatrix{
X \ar[r]^P  \ar[d]_{\iota_X} \ar@{}[dr]|{\iota_P} & X \ar[d]^{\iota_X} \\
X^* \ar[r]  \ar[d]_{u^\sharp} \ar@{}[dr]|{\bar{u}^\sharp} & X^* \ar[d]^{u^\sharp}  \\
X' \ar[r]_{P'} & X' .}}}
\]
\item \label{item:adj3} For every horizontal endomorphism map
$(F,\phi) \colon (X,P) \rightarrow (Y,Q)$, a horizontal
monad map $(F^*, \phi^*) \colon (X^*,P^*) \rightarrow (Y^*,Q^*)$.
\item \label{item:adj4} For every horizontal endomorphism map  $(F, \phi) : (X,P) \rightarrow (Y,Q)$,
a universal endomorphism square
\[
\xymatrix @C=8ex{
(X,P) \ar[r]^{(F,\phi)} \ar[d]_{(\iota_X, \iota_P)}  \ar@{}[dr]|{\iota_{(F,\phi)}} & (Y,Q) \ar[d]^{(\iota_Y, \iota_Q)} \\
(X^*,P^*) \ar[r]_{(F^*,\phi^*)} & (Y^*,Q^*) . }
\]
Universality means that for every endomorphism square
\begin{equation}
\label{equ:myalpha}
\vcenter{\hbox{
\xymatrix @C =8ex
{ (X,P) \ar[r]^-{(F,\phi)} \ar[d]_-{(u,\bar{u})} \ar@{}[dr]|-{\alpha} & (Y,Q) \ar[d]^-{(v, \bar{v})} \\
(X',P') \ar[r]_-{(F', \phi')} & (Y',Q'), } }}
\end{equation}
where $(X',P')$, $(Y',Q')$ are monads and $(F, \phi') : (X',P')
\rightarrow (Y',Q')$ is a horizontal monad map,  there exists a
unique monad square
\[
\xymatrix @C=8ex {(X^*, P^*)  \ar[r]^-{(F^*, \phi^*)} \ar[d]_-{(u^\sharp,\bar{u}^\sharp)} \ar@{}[dr]|{\alpha^\sharp}  & (Y^*, Q^*)
\ar[d]^-{(v^\sharp, \bar{v}^\sharp)} \\
(X',P') \ar[r]_-{(F', \phi')} & (Y',Q') }
\]
 such that
\[
\vcenter{\hbox{
\xymatrix@C=8ex  {
(X,P) \ar[r]^{(F,\phi)} \ar[d]_-{(u,\bar{u}) }   \ar@{}[dr]|-{\alpha} & (Y,Q) \ar[d]^-{(v,\bar{v})} \\
(X',P')  \ar[r]_{(F',\phi')} & (Y',Q') }  }}  \quad = \quad
\vcenter{\hbox{
 \xymatrix@C=8ex @R = 1cm {
(X,P) \ar[r]^{(F,\phi)} \ar[d]_{(\iota_X,\iota_P)} \ar@{}[dr]|{\iota_{(F,\phi)}}  & (Y,Q) \ar[d]^{(\iota_Y,\iota_Q)}   \\
(X^*,P^*) \ar[r] \ar[d]_{(u^\sharp, \bar{u}^\sharp)} \ar@{}[dr]|{\alpha^\sharp} & (Y^*, Q^*)
\ar[d]^{(v^\sharp, \bar{v}^\sharp)} \\
(X', P')  \ar[r]_{(F', \phi') } & (Y',Q').  }}} \medskip
\]
\item[$(\ast)$] The assignments in~\eqref{item:adj1} and~\eqref{item:adj3} give a functor
\[
    (-)^* : \big( \Obj_ {\End(\bbC)} \,  , \Hor_{\End(\bbC)} \big)  \rightarrow
              \big( \Obj_{\Mnd(\bbC)}  \, , \Hor_{\Mnd(\bbC)} \big)
\]
and the assignments in~\eqref{item:adj2} and~\eqref{item:adj4} give a functor
\[
\iota : \big( \Obj_ {\End(\bbC)} \,  , \Hor_{\End(\bbC)} \big)  \rightarrow
              \big( \Ver_{\End(\bbC)}  \, , \Sq_{\End(\bbC)} \big) \, .
\]
\end{enumerate}
Note that the data and the universality in~\eqref{item:adj2} actually follow from the data and the universality
in~\eqref{item:adj4} by taking $(F,\phi)$ to be the horizontal identity on an endomorphism~$(X,P)$.
\end{rmk}

A necessary condition for $U : \Mnd(\bbC) \rightarrow \End(\bbC)$ to have a vertical left adjoint is
that its vertical part
\begin{equation}
\label{equ:verticaladj}
U_0 : \Mnd(\bbC)_0 \rightarrow \End(\bbC)_0
\end{equation}
has a left adjoint. Indeed,  this is precisely what items~\eqref{item:adj1} and~\eqref{item:adj2}
of Remark~\ref{thm:charadj} amount to. Here, $\End(\bbC)_0$ denotes the category of endomorphisms and vertical endomorphism maps and $\Mnd(\bbC)_0$
denotes the category of monads and vertical monad maps.

\begin{examp} Let $\catE$ be a category with finite limits. The functor in~\eqref{equ:verticaladj}
for the double category~$\Sp_\catE$ is the forgetful functor $U_0 : \Cat_\catE  \rightarrow  \Grph_\catE$
mapping a category in $\catE$ to its underlying graph.
\end{examp}

\begin{examp} Let $\catE$ be a locally cartesian closed category.
The functor in~\eqref{equ:verticaladj} for the double category $\Poly_\catE$  is the forgetful
functor~$U_0 : \PolyMnd_\catE \rightarrow \PolyEnd_\catE$ mapping a polynomial monad to its
underlying endofunctor.
\end{examp}

\section{Monads in a framed bicategory}
\label{sec:monfb}

We now proceed to establish some properties of the double category~$\Mnd(\bbC)$ under the assumption that~$\bbC$ is
a framed bicategory, leading to our main theorem (Theorem~\ref{thm:freemndfrmbicat} below), which provides conditions
for~$\bbC$ to admit the construction of free monads. We begin by recalling from~\cite{ShulmanM:frabmf} the definition
of a framed bicategory and some useful facts.

\subsection*{Framed bicategories.} For a double category $\bbC$, the functor
\[
(\partial_0, \partial_1)  : \bbC_1
\rightarrow \bbC_0 \times \bbC_0 \, ,
\]
mapping a horizontal
arrow~$F \colon X \rightarrow Y$ to~$(X,Y)$ and a square as
in~\eqref{equ:typicalsquare} to~$(u,v) : (X,Y) \rightarrow (X',Y')$,
is a Grothendieck fibration if and only if it is a Grothendieck opfibration~\cite[Theorem~4.1]{ShulmanM:frabmf}.
When these conditions hold, the double category~$\bbC$ is said to be a framed
bicategory~\cite[Definition~4.2]{ShulmanM:frabmf}.
As explained in~\cite[Examples~4.4]{ShulmanM:frabmf}
and~\cite[Proposition~3.6]{GambinoN:polfpm},
the double categories~$\Span_\catE$ and~$\Poly_\catE$
are framed bicategories.

\begin{lem}[Shulman] \label{thm:shulman1} If $\bbC$ is a framed bicategory,
for every vertical arrow $u \colon X \rightarrow~X'$ there exist
horizontal arrows $\ladj{u} : X \rightarrow X'$ and $\pbk{u} : X'
\rightarrow X$ together with squares
\[
\xymatrix{
X \ar[r]^{\ladj{u}} \ar[d]_u   \ar@{}[dr]|{\alpha_u} & X' \ar@{=}[d]  \\
X'  \ar@{=}[r] & X'} \qquad
\xymatrix{
X' \ar@{=} \ar@{=}[d] \ar[r]^{\pbk{u}} \ar@{}[dr]|{\beta_u}  & X \ar[d]^{u} \\
X' \ar@{=}[r] & X' } \qquad
\xymatrix{
X \ar@{=}[r]  \ar[d]_{u} \ar@{}[dr]|{\gamma_u}  & X \ar@{=}[d] \\
X' \ar[r]_{\pbk{u}} & X } \qquad
\xymatrix{
X \ar@{=}[r] \ar@{=}[d] \ar@{}[dr]|{\delta_u}  & X \ar[d]^u  \\
X \ar[r]_{\ladj{u}} & X' }
\]
satisfying the equalities

\[
\xycenter{
X \ar@{=}[r]  \ar@{=}[d] \ar@{}[dr]|{\delta_u} & X   \ar[d]^{u}  \\
X  \ar[d]_{u}  \ar[r]  \ar@{}[dr]|{\alpha_u} &  X'  \ar@{=}[d] \\
X' \ar@{=}[r] &  X'}  \quad =  \quad
\xycenter{
X \ar[d]_{u} \ar@{=}[r]   & X  \ar[d]^{u}\\
X' \ar@{=}[r] & X' }  \quad = \quad
\xycenter{
X  \ar[d]_{u} \ar@{=}[r] \ar@{}[dr]|{\gamma_u} & X \ar@{=}[d] \\
X' \ar@{=}[d] \ar[r]  \ar@{}[dr]|{\beta_u} & X  \ar[d]^{u} \\
X' \ar@{=}[r]& X' ,}
\]

\smallskip

\[
\xycenter{
X \ar@{=}[r] \ar@{=}[d]   \ar@{}[dr]|{\delta_u} & X \ar[r]   \ar[d] \ar@{}[dr]|{\alpha_u} &  X' \ar@{=}[d]  \\
X \ar[r]  & X' \ar@{=}[r]  & X' } \quad = \quad
\xycenter{
X \ar[r]^{\ladj{u}}   \ar@{=}[d] & X' \ar@{=}[d] \\
X \ar[r]_{\ladj{u}} & X'}
\]

and

\[
\xycenter{
X'  \ar[r]  \ar@{=}[d] \ar@{}[dr]|{\beta_u}  & X   \ar[d] \ar@{=}[r]  \ar@{}[dr]|{\gamma_u} &  X \ar@{=}[d]  \\
X'   \ar@{=}[r]    & X' \ar[r]  & X } \quad = \quad
\xycenter{
X' \ar[r]^{\pbk{u}} \ar@{=}[d] & X \ar@{=}[d] \\
X' \ar[r]_{\pbk{u}} & X.}
\]
\end{lem}

\begin{proof} See~\cite[Theorem~4.1]{ShulmanM:frabmf}.
\end{proof}

Lemma~\ref{thm:shulman1} can be expressed equivalently by saying
that every vertical arrow~$u$ in~$\bbC$ has an orthogonal companion
$\ladj{u}$ and an orthogonal adjoint $\pbk{u}$ in the terminology
of~\cite{GrandisM:adjdc}.

\begin{lem}[Shulman] \label{thm:shulman2} Let $\bbC$ be a framed bicategory.
Let $u : X \rightarrow X'$ be a vertical arrow
in $\bbC$. If we define
\[
\xycenter{
X  \ar@{=}[d] \ar@{=}[rr]  \ar@{}[drr]|{\eta_u} & &  X \ar@{=}[d] \\
X \ar[r]_{\ladj{u}} & X' \ar[r]_{\pbk{u}} & X }   \quad \defeq \quad
\xycenter{
X  \ar@{=}[d] \ar@{=}[r] \ar@{}[dr]|{\delta_u} & X \ar@{=}[r] \ar[d] \ar@{}[dr]|{\gamma_u} &  X\ar@{=}[d] \\
X \ar[r]_{\ladj{u}} & X' \ar[r]_{\pbk{u}} & X }
\]
and
\[
\xycenter{
X' \ar[r]^{\pbk{u}}  \ar@{=}[d]  \ar@{}[drr]|{\varepsilon_u} &  X \ar[r]^{\ladj{u}}      & X' \ar@{=}[d]  \\
X' \ar@{=}[rr] &               & X' } \quad \defeq \quad
\xycenter{
X' \ar[r]^{\pbk{u}}  \ar@{=}[d]  \ar@{}[dr]|{\beta_u} &  X \ar[r]^{\ladj{u}} \ar[d]  \ar@{}[dr]|{\alpha_u} & X' \ar@{=}[d]  \\
X' \ar@{=}[r] &         X' \ar@{=}[r]      & X', }\smallskip
\]
then the following versions of the triangle identities hold:
\[
\xycenter{
X'  \ar@{=}[d] \ar[r]^{\pbk{u}} & X \ar@{=}[r]  \ar@{=}[d] \ar@{}[drr]|{\eta_u}  & X \ar@{=}[r] & X \ar@{=}[d] \\
X'   \ar@{=}[d] \ar[r]  \ar@{}[drr]|{\varepsilon_u}  & X \ar[r] & X' \ar@{=}[d] \ar[r]  & X \ar@{=}[d] \\
X' \ar@{=}[rr] &  & X' \ar[r]_{\pbk{u}} & X} \quad = \quad
\xycenter{
X' \ar@{=}[d]  \ar[r]^{\pbk{u}} & X \ar@{=}[d] \\
X' \ar[r]_{\pbk{u}} & X , }
\]

\[
\xycenter{
X \ar@{=}[rr] \ar@{=}[d]  \ar@{}[drr]|{\eta_u} &  & X \ar[r]^{\ladj{u}} \ar@{=}[d]  & X' \ar@{=}[d] \\
X   \ar@{=}[d] \ar[r]  & X' \ar[r]  \ar@{=}[d]  \ar@{}[drr]|{\varepsilon_u} & X \ar[r]  & X \ar@{=}[d] \\
X   \ar[r]_{\ladj{u}} & X \ar@{=}[rr]  &  & X} \quad = \quad
\xycenter{
X \ar@{=}[d]  \ar[r]^{\ladj{u}} & X' \ar@{=}[d] \\
X \ar[r]_{\ladj{u}} & X' . }
\]
\end{lem}

\begin{proof} See~\cite[Proposition~5.3]{ShulmanM:frabmf}.
\end{proof}

\subsection*{Monads in framed bicategories.} Let $\bbC$ be a double category. We have the diagram
\begin{equation}
\label{equ:general0}
\xycenter{
\Mnd (\bbC)_0 \ar[rr]^{U_0} \ar[dr]_{\partial_M} & &  \End ( \bbC)_0 \ar[dl]^{\partial_E} \\
 & \bbC_0 , & }
 \end{equation}
 where $\partial_E$ and $\partial_M$ send an endomorphism and a monad, respectively, to their
 underlying object and $U_0$ is the vertical part of the forgetful double functor
 $U$ of Definition~\ref{thm:deffrem}.

\begin{prop} \label{thm:grothfib}
If $\bbC$ is a framed bicategory, the functors
\[
\partial_E : \End(\bbC)_0 \rightarrow \bbC_0 \, , \quad
\partial_M \colon \Mnd(\bbC)_0 \rightarrow \bbC_0
\]
are Grothendieck fibrations and the functor $U_0 \colon \Mnd(\bbC)_0
\rightarrow~\End(\bbC)_0$ is a fibered functor relatively to these
fibrations.
\end{prop}

\begin{proof} Writing  $\Delta : \bbC_0 \rightarrow \bbC_0 \times \bbC_0$ for  the diagonal functor,
the functor~$\partial_E$  fits into the pullback diagram
\[
\xymatrix @C=8ex {
\End(\bbC)_0 \ar[r] \ar[d]_{\partial_E}    & \bbC_1 \ar[d]^{(\partial_0,\partial_1)} \\
\bbC_0 \ar[r]_-{\Delta}  & \, \bbC_0 \times \bbC_0.}
\]
We then have that $\partial_E$ is a Grothendieck fibration since it
is a pullback of $(\partial_0, \partial_1)$, which is a Grothendieck
fibration by the hypothesis that $\bbC$ is a framed bicategory.
Using Lemmas~\ref{thm:shulman1} and~\ref{thm:shulman2}, we can
 define explicitly a base change operation for the
Grothendieck fibration $\partial_E$, as follows. Let~$u : X \rightarrow X'$ be a
map in $\bbC_0$ and~$(X',P')$ an endomorphism in $\bbC$. The base change of
$(X',P')$ along $u$ is defined to be the endomorphism~$(X,P)$, where $P : X
\rightarrow X$ is the composite
\[
\xymatrix{
X \ar[r]^{\ladj{u}} & X' \ar[r]^{P'} & X' \ar[r]^{\pbk{u}} & X } \, .
 \]
The required cartesian morphism from $(X,P)$ to $(X',P')$ in $\End(\bbC)_0$ (\myemph{i.e.}~the
cartesian lift of $u$)  is given
by the vertical endomorphism map~$(u, \bar{u}) : (X,P) \rightarrow (X', P')$, where $\bar{u}$ is the square
\[
\xymatrix{
X  \ar[r]^{\ladj{u}} \ar[d]_u  \ar@{}[dr]|{\alpha_u} & X' \ar[r]^{P'}  \ar@{=}[d] & X' \ar[r]^{\pbk{u}}   \ar@{=}[d]   \ar@{}[dr]|{\beta_u} & X \ar[d]^u \\
X' \ar@{=}[r] & X'  \ar[r]_{P'}  & X' \ar@{=}[r]      & X'.}
\]
The verification of the required universal property is
straightforward.

To show that~$\partial_M$ is  a Grothendieck
fibration, we first observe that if $(X',P')$ is a monad, then~$(X,P)$ inherits a monad structure: its multiplication is the square
\[
\xymatrix{
X  \ar[r]^{\ladj{u}}  \ar@{=}[d] & X' \ar[r]^{P'} \ar@{=}[d]  & X' \ar[r]^{\pbk{u}} \ar@{=}[d] \ar@{}[drr]|{\varepsilon_u}
& X \ar[r]^{\ladj{u}}
& X' \ar[r]^{P'} \ar@{=}[d] & X' \ar[r]^{\pbk{u}} \ar@{=}[d] & X \ar@{=}[d]  \\
X \ar[r] \ar@{=}[d] & X' \ar[r]_{P'}  \ar@{=}[d]
\ar@{}[drrrr]|{\mu}
 & X' \ar@{=}[rr] &  & X' \ar[r]_{P'}  & X' \ar[r] \ar@{=}[d] & X \ar@{=}[d] \\
X \ar[r]_{\ladj{u}} & X' \ar[rrrr]_{P'} &  & & &   X' \ar[r]_{\pbk{u}} & X }
\]
and its unit is the square
\[
\xymatrix{
X \ar@{=}[r]  \ar@{=}[d] \ar@{}[drrr]|{\eta_u} & X \ar@{=}[r] & X \ar@{=}[r] & X \ar@{=}[d] \\
X \ar[r]^{\ladj{u}}   \ar@{=}[d] & X' \ar@{=}[r] \ar@{=}[d]  \ar@{}[dr]|{\eta} & X' \ar[r]^{\pbk{u}} \ar@{=}[d] & X \ar@{=}[d]  \\
X \ar[r]_{\ladj{u}}  & X' \ar[r]_{P'} & X' \ar[r]_{\pbk{u}}  & X. }
\]
The monad axioms are easily verified. Now it only remains to verify that the cartesian
lift $(u, \bar{u})$ is a vertical monad map and that it is cartesian for $\partial_M$. This
verification is straightforward, using Lemmas~\ref{thm:shulman1} and~\ref{thm:shulman2}.
This also shows that $U_0$ is fibered as claimed.
\end{proof}

\begin{lem} \label{thm:verhorbij} Let $(X,P)$ and $(X',P')$ be endomorphisms in a framed bicategory
$\bbC$. There is a bijection between vertical endomorphism  maps
$(u, \bar{u}) : (X, P) \rightarrow (X',P')$ and horizontal
endomorphism maps of the form $(\pbk{u}, \phi) : (X', P')
\rightarrow (X,P)$, which restricts to a bijection between vertical
monad  maps  and horizontal monad maps when $(X,P)$ and $(X',P')$
are monads.
\end{lem}

\begin{proof}
For a vertical endomorphism map $(u, \bar{u}) : (X, P)
\rightarrow (X',P')$, define the horizontal endomorphism map
$(\pbk{u}, \phi_u) : (X', P') \rightarrow (X,P)$ by letting $\phi_u$
be the square
\[
\xycenter{ X' \ar[r]^{\pbk{u}} \ar@{=}[d]  \ar@{}[drr]|{\phi_u} & X
\ar[r]^P & X \ar@{=}[d] \\
X' \ar[r]_{P'}  & X' \ar[r]_{\pbk{u}}  & X } \quad \defeq \quad
\xycenter{ X' \ar[r]^{\pbk{u}} \ar@{=}[d] \ar@{}[dr]|{\beta_u} & X
\ar[r]^P \ar[d] \ar@{}[dr]|{\bar{u}} &
X \ar@{=}[r] \ar[d] \ar@{}[dr]|{\gamma_u} & X \ar@{=}[d] \\
X' \ar@{=}[r] & X' \ar[r]_{P'}  & X' \ar[r]_{\pbk{u}} & X .}
\]
In the other direction, given a horizontal endomorphism map
$(\pbk{u}, \phi) \colon (X', P') \rightarrow~(X,P)$, define
the vertical endomorphism map $(u, \bar{u}_\phi) : (X,P) \rightarrow
(X',P')$ by letting $\bar{u}_\phi$ be the square
\[
\xycenter{ X \ar[r]^P \ar[d]_u  \ar@{}[dr]|{\bar{u}_\phi} & X \ar[d]^u \\
X' \ar[r]_{P'} & X' } \quad \defeq \quad \xycenter{
X \ar@{=}[r] \ar[d]_u \ar@{}[dr]|{\gamma_u} & X \ar[r]^P \ar@{=}[d] & X \ar@{=}[d] \\
X' \ar[r] \ar@{=}[d] \ar@{}[drr]|{\phi} & X \ar[r] & X \ar@{=}[d] \\
X' \ar[r] \ar@{=}[d] & X' \ar[r]  \ar@{=}[d] \ar@{}[dr]|{\beta_u} & X \ar[d]^u \\
X' \ar[r]_{P'} & X' \ar@{=}[r] & X' .}
\]
Using Lemma~\ref{thm:shulman1}, it is possible to show that these functions are mutually inverse,
that~$(\pbk{u}, \phi_u)$ is a horizontal monad map if~$(u, \bar{u})$ is a
vertical monad map, and that~$(u, \bar{u}_\phi)$ is a vertical monad
map if~$(\pbk{u}, \phi)$ is a horizontal monad map.
\end{proof}

Let us point out that the bijection defined in the proof of Lemma ~\ref{thm:verhorbij} is an example
of a cofolding in the sense of~\cite[Definition~3.16]{FioreT:pseapd}.

\subsection*{Free monads in a framed bicategory.}
We now consider the construction of free monads in a framed bicategory $\bbC$. Since 
the functor~$U_0$ in~\eqref{equ:general0} is fibered, a sufficient condition for it to have
a left adjoint is that each of its fibers has a left adjoint. In this case, the free monad on an
endomorphism~$(X,P)$ has the form~$(X,P^*)$ and 
the component of the unit~$(\iota_X, \iota_P) : (X,P) \rightarrow
(X,P^*)$ is a vertical endomorphism map of the form~$(1_X, \iota_P) :
(X, P) \rightarrow (X, P^*)$, where~$\iota_P$ is a square of the form
\[
\xymatrix{
X \ar[r]^{P} \ar@{=}[d]  \ar@{}[dr]|{\iota_P}& X \ar@{=}[d] \\
X \ar[r]_{P^*} & X.}
\]
The universal property in the fiber asserts that for every
endomorphism square of the form
\[
\xymatrix{
X \ar[r]^{P} \ar@{=}[d]  \ar@{}[dr]|{\alpha}& X \ar@{=}[d] \\
X \ar[r]_{P'} & X,}
\]
where~$(X,P')$ is a monad,  there exists a unique monad square of the
form
\[
\xymatrix{
X \ar[r]^{P^*} \ar@{=}[d]  \ar@{}[dr]|{\alpha^\sharp}& X \ar@{=}[d] \\
X \ar[r]_{P'} & X}
\]
such that
\[
\xycenter{
X \ar[r]^{P} \ar@{=}[d]  \ar@{}[dr]|{\alpha}& X \ar@{=}[d] \\
X \ar[r]_{P'} & X} \quad = \quad \xycenter{
X \ar[r]^{P} \ar@{=}[d]  \ar@{}[dr]|{\iota_P}& X \ar@{=}[d] \\
X \ar[r]  \ar@{=}[d]  \ar@{}[dr]|{\alpha^\sharp}& X \ar@{=}[d] \\
X \ar[r]_{P'} & X.}
\]
The universal property in the fiber implies a more general universal property,
with respect to general monads (and not just monads
with~$X$ as underlying object) and general vertical endomorphism maps
(and not just the special ones considered above), as in item~\eqref{item:adj2}
of Remark~\ref{thm:charadj}.  Note, however, that  the left adjoint to 
$U_0$ constructed from the left adjoints to its fibers need not be a fibered left
adjoint, since the so-called Beck-Chevalley conditions are not necessarily
satisfied~\cite[\S1.8]{JacobsB:catltt}.

\begin{examp} Let us consider the framed bicategory $\Sp_\catE$ associated to a category $\catE$  with finite limits. The
diagram in~\eqref{equ:general0} becomes
\[
\xymatrix{
\Cat_\catE \ar[rr]^{U_0} \ar[dr]_{\partial_M} &  & \Grph_\catE   \ar[dl]^{\partial_E} \\
  & \catE, & }
\]
where $\partial_E$ sends a graph to its object of vertices and $\partial_M$ sends a small category to 
its object of objects. Since $\Sp_\catE$ is a framed bicategory, the preceding remarks reduce to the familiar fact that the free category on a graph has the object of vertices of the graph as its object of objects.
\end{examp}

\begin{examp} Let us consider the framed bicategory $\Poly_\catE$ associated to a locally cartesian closed
category~$\catE$ with finite disjoint coproducts. For $\Poly_\catE$, the diagram in~\eqref{equ:general0} amounts to:
\[
\xymatrix{
\PolyMnd_\catE \ar[rr]^{U_0}  \ar[dr]_{\partial_M} &  & \PolyEnd_\catE  \ar[dl]^{\partial_E} \\
&  \catE .& }
\]
In this case, the remarks above amount to the fact, exploited in the proof of~\cite[Corollary~4.7]{GambinoN:polfpm}, that
to prove the universal property of the free monad on a polynomial endofunctor with respect to maps in~$\PolyEnd_\catE$,
it is sufficient to check it with respect to a special class of them.
\end{examp}

Theorem~\ref{thm:freemndfrmbicat}, which is our main result, gives sufficient conditions for a
framed bicategory to admit the construction of free monads, facilitating the verification of this property in
our examples. Recall that, for a double category $\bbC$, we write $\bbC_0$ for its category of
objects and vertical arrows and $\bbC_1$ for its category of horizontal arrows and squares.

\begin{thm} \label{thm:freemndfrmbicat}
Let $\bbC$ be a framed bicategory such that the category $\bbC_1$ 
has equalizers and the source and target functors $\partial_0, \partial_1 :
\bbC_1 \rightarrow \bbC_0$ preserve them.  If the horizontal 2-category of $\bbC$ has local coproducts
and admits the construction of free monads, then $\bbC$ admits the construction of
free monads.
\end{thm}

The proof of Theorem~\ref{thm:freemndfrmbicat} is given in Section~\ref{sec:technical}. Here,
instead, we apply it to our two running examples.

\begin{prop} \label{thm:freemndexamp} \, \hfill
\begin{enumerate}[(i)]
\item If $\catE$ is a pretopos with parametrized list objects,
the double category $\Sp_\catE$ admits the construction of free monads.
\item If $\catE$ is a  locally cartesian closed category  with disjoint coproducts and
W-types, the double category $\Poly_\catE$ admits the construction of free monads.
\end{enumerate}
\end{prop}

\begin{proof} For both (i) and (ii), we apply Theorem~\ref{thm:freemndfrmbicat}.
For (i), the hypotheses on equalizers are verified because in this case the
category $\bbC_1$ is a category of internal presheaves, in which equalizers 
exist and are computed pointwise. For (ii), the hypotheses on equalizers 
are also satisfied, since pullbacks preserve equalizers. For both items, 
the existence of free monads in the horizontal 2-categories is established in
Proposition~\ref{thm:key2cat}.
\end{proof}

\section{Proof of the Main Theorem}
\label{sec:technical}

Let $\bbC$ be a double category satisfying the hypotheses of Theorem~\ref{thm:freemndfrmbicat}.
We use the characterization of free
monads in a 2-category given in Theorem~\ref{thm:frem2cat} to exhibit the data listed in
Remark~\ref{thm:charadj}. For items~\eqref{item:adj1} and~\eqref{item:adj2} of Remark~\ref{thm:charadj}, let $(X,P)$ be
an endomorphism. By the existence of free monads in $\mathcal{H}_\bbC$, we have a monad
$(X,P^*)$ and a square
\[
\xymatrix{
X \ar[r]^P \ar@{=}[d] \ar@{}[dr]|{\iota_P} & X \ar@{=}[d] \\
X \ar[r]_{P^*} & X }
\]
satisfying the equivalent conditions in items~(i) and~(ii) of Theorem~\ref{thm:frem2cat} in $\mathcal{H}_\bbC$.
We then obtain a vertical endomorphism map $(1_X, \iota_P) : (X,P)
\rightarrow (X,P^*)$. We need to show that~$(1_X, \iota_P)$ enjoys the required universal
property. For this, let us consider a vertical endomorphism map~$(u,
\bar{u}) : (X,P) \rightarrow (X', P')$, where~$(X',P')$ is a monad.
Here, $\bar{u}$ is a square of the form
\[
\xymatrix{ X \ar[r]^P \ar[d]_u  \ar@{}[dr]|{\bar{u}}  & X \ar[d]^u \\
X' \ar[r]_{P'} & X'.}
\]
By the cofolding bijection defined in the proof of Lemma~\ref{thm:verhorbij},
we have a horizontal endomorphism map $(\pbk{u}, \phi_u) : (X', P')
\rightarrow (X,P)$, where $\phi_u$ is a square of the form
\[
\xymatrix{
X' \ar[r]^{\pbk{u}} \ar@{=}[d] \ar@{}[drr]|{\phi_u} & X \ar[r]^P & X \ar@{=}[d] \\
X' \ar[r]_{P'} & X' \ar[r]_{\pbk{u}} & X .}
\]
By the universal property in item~(i) of Theorem~\ref{thm:frem2cat} for $(X,P^*)$, there exists a unique square
\[
\xymatrix{
X' \ar[r]^{\pbk{u}} \ar@{=}[d] \ar@{}[drr]|{\phi_u^\sharp} & X \ar[r]^{P^*} & X \ar@{=}[d] \\
X' \ar[r]_{P'} & X' \ar[r]_{\pbk{u}} & X }
\]
such that $(\pbk{u}, \phi_u^\sharp) : (X', P') \rightarrow (X,P^*)$ is
a horizontal monad map and
\[
\xycenter{
X' \ar[r]^{\pbk{u}} \ar@{=}[d] \ar@{}[drr]|{\phi_u} & X \ar[r]^P & X \ar@{=}[d] \\
X' \ar[r]_{P'}  & X' \ar[r]_{\pbk{u}} & X }
\quad = \quad
\xycenter{
X' \ar[r]^{\pbk{u}} \ar@{=}[d]  & X \ar[r]^P \ar@{=}[d] \ar@{}[dr]|{\iota_P} & X \ar@{=}[d] \\
X' \ar[r] \ar@{=}[d] \ar@{}[drr]|{\phi_u^\sharp} & X \ar[r] & X \ar@{=}[d] \\
X' \ar[r]_{P'} & X' \ar[r]_{\pbk{u}} & X  .}
\]
Using again the cofolding bijection of Lemma~\ref{thm:verhorbij}, we obtain the vertical monad
morphism $(u, \bar{u}^\sharp) : (X,P^*) \rightarrow (X', P')$ that factors $(u, \bar{u})$ through
$(1_X, \iota_P)$, as required. By the definition of the bijection and Theorem~\ref{thm:frem2cat},
the square $\bar{u}^\sharp$ satisfies the equations
\begin{equation} \label{equ:transposefirst}
\xycenter{
X \ar@{=}[r] \ar@{=}[d]  \ar@{}[dr]|{\eta_{P^*}} &  X \ar@{=}[d] \\
X \ar[r]   \ar[d]_{u}  \ar@{}[dr]|{\bar{u}^\sharp} & X  \ar[d]^{u} \\
X' \ar[r]_{P'} & X' }  \quad = \quad \xycenter{
X \ar@{=}[r]  \ar[d]_{u} &  X \ar[d]^{u}  \\
X' \ar@{=}[r]  \ar@{=}[d]   \ar@{}[dr]|{\eta_{P'}}  & X'  \ar@{=}[d] \\
X' \ar[r]_{P'} & X' }
\end{equation}
and
\begin{equation}
\label{equ:transposesecond} \xycenter{
X \ar[r]^{P^*}  \ar@{=}[d]  \ar@{}[drr]|{\nu_{P^*}} & X   \ar[r] ^{P} & X \ar@{=}[d]  \\
X \ar[rr] \ar[d]  \ar@{}[drr]|{\bar{u}^\sharp} & & X \ar[d]  \\
X' \ar[rr]_{P'}                  & & Y } \quad = \quad \xycenter{ X
\ar[r]^{P^*} \ar[d]   \ar@{}[dr]|{\bar{u}^\sharp}  & X \ar[r]^{P}
\ar[d] \ar@{}[dr]|{\bar{u}} &
 X \ar[d]  \\
X' \ar[r]         \ar@{=}[d]  \ar@{}[drr]|{\mu_{P'}}   & X' \ar[r]       &    X' \ar@{=}[d] \\
X' \ar[rr]_{P'}      &                        &  X' ,}
\end{equation}
where the square $\nu_{P^*}$ is defined by
\[
\xycenter{
X \ar[r]^{P^*} \ar@{=}[d] \ar@{}[drr]|{\nu_{P^*}} & X \ar[r]^{P} & X  \ar@{=}[d] \\
X \ar[rr]_{P^*} & & X } \quad \defeq \quad
\xycenter{
X \ar[r]^{P^*} \ar@{=}[d]  & X \ar[r]^{P} \ar@{}[dr]|{\iota_P} \ar@{=}[d] & X  \ar@{=}[d] \\
X \ar[r] \ar@{=}[d] \ar@{}[drr]|{\mu_{P^*}} & X \ar[r]  & X \ar@{=}[d] \\
X \ar[rr]_{P^*} & & X .}
\]
For item~\eqref{item:adj3} of Remark~\ref{thm:charadj}, let $(F, \phi) : (X,P) \rightarrow (Y, Q)$ be a horizontal endomorphism map.
Exploiting the universal property in item~(i) of Theorem~\ref{thm:frem2cat} for $(Y,Q^*)$,
we define
\[
\xymatrix{ X \ar[r]^F \ar@{=}[d] \ar@{}[drr]|{\phi^*} & Y
\ar[r]^{Q^*} & Y
\ar@{=}[d] \\
X \ar[r]_{P^*} & X \ar[r]_F & X}
\]
to be the unique square  such that $(F, \phi^*) : (X, P^*)
\rightarrow (Y, Q^*)$ is a horizontal monad map and
\begin{equation}
\label{equ:phistar}
\xycenter{ X \ar[r]^F \ar@{=}[d] \ar@{}[drr]|{\phi} & Y \ar[r]^Q & Y \ar@{=}[d] \\
X \ar[r]  \ar@{=}[d] \ar@{}[dr]|{\iota_{P}} & X \ar[r] \ar@{=}[d] & Y \ar@{=}[d] \\
X \ar[r]_{P^*} & X \ar[r]_{F}  & Y } \quad = \quad \xycenter{ X \ar[r]^F
\ar@{=}[d] & Y \ar[r]^Q  \ar@{=}[d] \ar@{}[dr]|{\iota_Q}
& Y \ar@{=}[d] \\
 X \ar[r] \ar@{=}[d] \ar@{}[drr]|{\phi^*} & Y \ar[r]  & Y
\ar@{=}[d] \\
X \ar[r]_{P^*} & X \ar[r]_F & X.}
\end{equation}
Observe that, by the fact that $(F, \phi^*)$ is a horizontal monad map, we have that
\begin{equation}
\label{equ:phistarfirst} \xycenter{
X \ar[r]^F \ar@{=}[d] & Y \ar@{=}[d]  \ar@{=}[r]  \ar@{}[dr]|{\eta_{Q^*}} & Y \ar@{=}[d] \\
X \ar@{=}[d] \ar[r]  \ar@{}[drr]|{\phi^*} & Y \ar[r] & Y \ar@{=}[d] \\
X \ar[r]_{P^*} & X \ar[r]_F & Y }  \quad = \quad
\xycenter{
X \ar@{=}[r] \ar@{=}[d] \ar@{}[dr]|{\eta_{P^*}} & X \ar@{=}[d] \ar[r]^F & Y \ar@{=}[d] \\
X \ar[r]_{P^*} & X \ar[r]_{F} & Y  }
\end{equation}
and
\begin{equation}
\label{equ:phistarsecond} \xycenter{
X \ar@{=}[d] \ar[r]^F & Y  \ar@{=}[d] \ar@{}[drr]|{\nu_{Q^*}} \ar[r]^{Q^*} & Y    \ar[r]^Q & Y  \ar@{=}[d] \\
X \ar@{=}[d] \ar[r] \ar@{}[drrr]|{\phi^*} & Y  \ar[rr] & & Y \ar@{=}[d] \\
X \ar[rr]_{P^*} & & X  \ar[r]_F &   Y} \quad = \quad \xycenter{
X \ar@{=}[d]\ar[r]^F & Y \ar@{}[d]|{\phi^*}\ar[r]^{Q^*} & Y \ar@{=}[d]\ar[r]^{Q} & Y \ar@{=}[d] \\
X \ar@{=}[d]\ar[r] & X \ar@{=}[d]\ar[r] & Y \ar@{}[d]|{\phi}\ar[r] & Y \ar@{=}[d]\\
X \ar[r] \ar@{=}[d]  \ar@{}[drr]|{\nu_{P^*}} & X \ar[r]  & X \ar[r]  \ar@{=}[d] & Y  \ar@{=}[d]  \\
X \ar[rr]_{P^*} & & X \ar[r]_{F} & Y .}
\end{equation}
In particular, \eqref{equ:phistarsecond} holds by the definitions of $\nu_{P^*}$ and $\nu_{Q^*}$, the
first axiom for a horizontal monad map and~\eqref{equ:phistar}. For item~\eqref{item:adj4} of
Remark~\ref{thm:charadj}, the required universal endomorphism square needs to have
the form
\[
\xymatrix{ (X,P) \ar[r]^{(F, \phi)} \ar[d]_{(1_X, \iota_P)}
\ar@{}[dr]|{\iota_{(F,\phi)}}
& (Y,Q)  \ar[d]^{(1_Y, \iota_Q)} \\
(X,P^*) \ar[r]_{(F, \phi^*)} & (Y, Q^*).}
\]
Therefore, $\iota_{(F,\phi)}$ has to be a square in $\bbC$ of the
form
\[
\xymatrix{ X \ar[r]^F \ar@{=}[d] \ar@{}[dr]|{\iota_{(F,\phi)}}  & Y \ar@{=}[d] \\
X \ar[r]_F & Y }
\]
and satisfy the equation
\begin{equation}
\label{equ:previous}
\xycenter{
X \ar[r]^F \ar@{=}[d] \ar@{}[drr]|{\phi} & Y \ar[r]^Q & Y \ar@{=}[d] \\
X \ar[r] \ar@{=}[d] \ar@{}[dr]|{\iota_P} & X \ar[r] \ar@{=}[d] \ar@{}[dr]|{\iota_{(F,\phi)}}  & Y \ar@{=}[d] \\
X \ar[r]_{P^*} & X \ar[r]_F & Y } \quad = \quad \xycenter{
X \ar[r]^F \ar@{=}[d]  \ar@{}[dr]|{\iota_{(F,\phi)}}  & Y \ar[r]^Q \ar@{=}[d] \ar@{}[dr]|{\iota_Q} & Y \ar@{=}[d] \\
X \ar[r] \ar@{=}[d] \ar@{}[drr]|{\phi^*} & Y \ar[r] & Y \ar@{=}[d] \\
X \ar[r]_{P^*} & X \ar[r]_F & Y .}
\end{equation}
We define $\iota_{(F,\phi)}$ to be the identity square on $F$, so
that~\eqref{equ:previous} above is verified by~\eqref{equ:phistar}.
To verify the universal property, we need to show that for an endomorphism square
\[
\xymatrix@C=8ex{
(X,P) \ar[r]^{(F,\phi)} \ar[d]_{(u, \bar{u})} \ar@{}[dr]|{\alpha} & (Y, Q) \ar[d]^{(v, \bar{v})} \\
(X', P') \ar[r]_{(F',\phi')} & (Y', Q'), }
\]
there exists a unique monad square
\[
\xymatrix@C=8ex {
(X,P^*) \ar[r]^{(F,\phi^*)} \ar[d]_{(u, \bar{u}^\sharp)}  \ar@{}[dr]|{\alpha^\sharp} & (Y,Q^*)
\ar[d]^{(v, \bar{v}^\sharp)}  \\
(X',P') \ar[r]_{(F',\phi')}  & (Y',Q') }
\]
satisfying
\begin{equation}
\label{equ:unitaxiom} \vcenter{\hbox{ \xymatrix@C=8ex{
(X,P) \ar[r]^{(F,\phi)} \ar[d]_{(u, \bar{u})} \ar@{}[dr]|{\alpha}  & (Y,Q) \ar[d]^{(v,\bar{v})} \\
(X',P') \ar[r]_{(F',\phi')} & (Y',Q') }}} \quad = \quad
\vcenter{\hbox{ \xymatrix @C=8ex{
(X,P) \ar[r]^{(F,\phi)} \ar[d]_{(1_X, \iota_P)} \ar@{}[dr]|{\iota_{(F,\phi)}}  & (Y,Q) \ar[d]^{(1_Y,\iota_Q)}  \\
(X,P^*) \ar[r] \ar[d]_{(u, \bar{u}^\sharp)}  \ar@{}[dr]|{\alpha^\sharp} & (Y,Q^*) \ar[d]^{(v, \bar{v}^\sharp)} \\
(X',P') \ar[r]_{(F',\phi')} & (Y', Q'). }}}
\end{equation}
First of all, observe that $\alpha$ is a square in $\bbC$ of the
form
\[
\xymatrix{ X \ar[r]^F \ar[d]_{u} \ar@{}[dr]|{\alpha} & Y \ar[d]^{v} \\
X' \ar[r]_{F'} & Y' }
\]
which satisfies the compatibility condition
\begin{equation}
\label{equ:hypothesisalpha} \xycenter{
X \ar[r]^F  \ar@{=}[d]   \ar@{}[drr]|{\phi}  & Y \ar[r]^Q            & Y \ar@{=}[d]  \\
X \ar[r]  \ar[d]  \ar@{}[dr]|{\bar{u}} & X \ar[d]  \ar[r]
\ar@{}[dr]|{\alpha} & Y
\ar[d]   \\
X' \ar[r]_{P'}             & X' \ar[r]_{F'}             & Y' } \quad
= \quad \xycenter{
X \ar[r]^F   \ar[d]   \ar@{}[dr]|{\alpha}  & Y \ar[r]^Q  \ar@{}[dr]|{\bar{v}}   \ar[d]       & Y \ar[d]  \\
X' \ar[r]  \ar@{=}[d]  \ar@{}[drr]|{\phi'} & Y'  \ar[r]   & Y' \ar@{=}[d] \\
X' \ar[r]_{P'}              & X' \ar[r]_{F'}             & Y'. }
\end{equation}
The required monad square $\alpha^\sharp$ has to be a square in $\bbC$ of the form
\[
\xymatrix{ X \ar[r]^F \ar[d]_{u} \ar@{}[dr]|{\alpha^\sharp} & Y \ar[d]^{v} \\
X' \ar[r]_{F'} & Y' }
\]
satisfying the compatibility condition
\begin{equation}
\label{equ:compatibilityfinal} \xycenter{
X \ar[r]^F  \ar@{=}[d]   \ar@{}[drr]|{\phi^*}  & Y \ar[r]^{Q^*}            & Y \ar@{=}[d]  \\
X \ar[r]  \ar[d]  \ar@{}[dr]|{\bar{u}^\sharp} & X \ar[d]  \ar[r]
\ar@{}[dr]|{\alpha^\sharp} & Y
\ar[d]   \\
X' \ar[r]_{P'}             & X' \ar[r]_{F'}             & Y' } \quad
= \quad \xycenter{
X \ar[r]^F   \ar[d]   \ar@{}[dr]|{\alpha^\sharp}  & Y \ar[r]^{Q^*}  \ar@{}[dr]|{\bar{v}^\sharp}   \ar[d]
& Y \ar[d]  \\
X' \ar[r]  \ar@{=}[d]  \ar@{}[drr]|{\phi'} & Y'  \ar[r]   & Y' \ar@{=}[d] \\
X' \ar[r]_{P'}              & X' \ar[r]_{F'}             & Y'. }
\end{equation}
We define $\alpha^\sharp \defeq  \alpha$, so that
Equation~\eqref{equ:unitaxiom} holds trivially, since $\iota_{(F,\phi)}$ is the identity.

To complete the verification of the universal property of $\iota_{(F, \phi)}$,
it only remains to show that Equation~\eqref{equ:compatibilityfinal} holds.
The idea is to consider the sub-horizontal arrow $E$ of $Q^*  F$ for which~\eqref{equ:compatibilityfinal}
and show that $E$ must be isomorphic to~$Q^*  F$. More precisely, let us define the
horizontal arrow $E : X \rightarrow Y$ via the following equalizer in
the category $\bbC_1$ of horizontal arrows and squares:
\begin{equation}
\label{equ:equalizer} \xymatrix@C=8ex{
 E \, \ar@{>->}[r]^-{\theta} & Q^*  F \ar@<1ex>[r]^{(\bar{u}^\sharp, \alpha) \, \phi^* }
 \ar@<-1ex>[r]_{\phi' \, (\alpha, \bar{v}^{\sharp})}
 & F'   P' \, .}
\end{equation}
Since the vertical boundaries of the squares in~\eqref{equ:compatibilityfinal} are equal, 
the assumption that the source and target functors $\partial_0, \partial_1 : \bbC_1
\rightarrow \bbC_0$ preserve equalizers implies that~$E$ is indeed a
horizontal arrow from $X$ to $Y$ and that $\theta$ has vertical boundaries given by
identity morphisms.  The commutativity of the equalizer
diagram in~\eqref{equ:equalizer} can be expressed as the equation
\begin{equation}
\label{equ:indhyp} \xycenter{
X \ar[rr]^E \ar@{=}[d] \ar@{}[drr]|{\theta} & & Y \ar@{=}[d] \\
X \ar@{=}[d] \ar[r] \ar@{}[drr]|{\phi^*} & Y \ar[r] & Y \ar@{=}[d] \\
X \ar[r] \ar[d] \ar@{}[dr]|{\bar{u}^\sharp} & X \ar[d] \ar[r] \ar@{}[dr]|{\alpha} & Y \ar[d] \\
X' \ar[r]_{P'} & X' \ar[r]_{F'} & Y' }
 \quad = \quad \xycenter{
X \ar[rr]^E \ar@{=}[d] \ar@{}[drr]|{\theta} & & Y \ar@{=}[d] \\
X \ar[r] \ar[d] \ar@{}[dr]|{\alpha} & Y \ar[d] \ar[r] \ar@{}[dr]|{\bar{v}^\sharp} & Y \ar[d] \\
X' \ar[r] \ar@{=}[d] \ar@{}[drr]|{\phi'} & Y' \ar[r] & Y' \ar@{=}[d] \\
X' \ar[r]_{P'} & X' \ar[r]_{F'} & Y' .}
\end{equation}
To prove Equation~\eqref{equ:compatibilityfinal} we show that
$\theta : E \rightarrow Q^* F$ is an isomorphism. For this, we
exploit the fact (observed in Remark~\ref{thm:usefulrmk})
that $Q^*  F : X \rightarrow Y$ is the initial
algebra for the endofunctor
\begin{eqnarray}
  \mathcal{H}_\bbC(X,Y) & \longrightarrow & \mathcal{H}_\bbC(X,Y)  \label{equ:endouseful} \\
  (-)  & \longmapsto & F  + Q (-) \, , \notag
\end{eqnarray}
where $\mathcal{H}_\bbC(X,Y)$ denotes the hom-category of horizontal
arrows from $X$ to $Y$ of the horizontal 2-category
$\mathcal{H}_\bbC$ of $\bbC$. Note that here we are using our assumption that
 $\mathcal{H}_\bbC$ has local coproducts.
By the initiality of~$Q^* F$, in
order to show that~$\theta : E \rightarrow Q^*  F$ is an
isomorphism, it is sufficient to show that~$E$ admits an algebra
structure for the endofunctor in~\eqref{equ:endouseful}. The
required algebra structure is given by the copair~$(\lambda, \rho) :
F + Q E \rightarrow E$, where $\lambda : F \rightarrow E$ and
$\rho : Q E \rightarrow E$ are determined, via the universal
property of the equalizer $E$, by the commutative diagrams
\begin{equation}
\label{equ:diag1} \xymatrix@C=8ex{
 F \, \ar[r]^-{\eta_{Q^*}  F} & Q^*  F \ar@<1ex>[r]^{(\bar{u}^\sharp, \alpha) \, \phi^* }
 \ar@<-1ex>[r]_{\phi' \, (\alpha, v^{\sharp})}
 & F'   P' }
\end{equation}
and
\begin{equation}
\label{equ:diag2} \xymatrix@C=8ex{
 Q  E \, \ar[r]^-{Q \, \theta} &
 Q  Q^*  F \ar[r]^-{\nu_{Q^*} \, F}
  & Q^*  F \ar@<1ex>[r]^{(\bar{u}^\sharp, \alpha) \, \phi^* } \ar@<-1ex>[r]_{\phi' \, (\alpha, v^{\sharp})}
 & F'   P' \, ,}
\end{equation}
respectively. It remains to show that the diagrams
in~\eqref{equ:diag1} and~\eqref{equ:diag2} commute. The
commutativity of~\eqref{equ:diag1} amounts to the equation
\begin{equation}
\label{Id} \xycenter{
X \ar@{=}[d]\ar[r]^F & Y  \ar@{=}[d] \ar@{}[dr]|{\eta_{Q^*}}  \ar@{=}[r] & Y \ar@{=}[d] \\
X \ar[r]  \ar@{=}[d]   \ar@{}[drr]|{\phi^*}  & Y \ar[r]       & Y \ar@{=}[d] \\
X \ar[r]  \ar[d]   \ar@{}[dr]|{\bar{u}^\sharp} & X \ar[d] \ar[r]
\ar@{}[dr]|{\alpha} & Y
\ar[d]   \\
X' \ar[r]_{P'}            & X'  \ar[r]_{F'}             & Y' } \quad
= \quad \xycenter{
X  \ar[r]^F \ar@{=}[d] & Y \ar@{=}[d]  \ar@{=}[r] \ar@{}[dr]|{\eta_{Q^*}} & Y \ar@{=}[d]  \\
X \ar[r]   \ar[d]   \ar@{}[dr]|{\alpha}  & Y \ar[r]   \ar@{}[dr]|{\bar{v}^\sharp}   \ar[d]       & Y \ar[d]  \\
X' \ar[r]  \ar@{=}[d]  \ar@{}[drr]|{\phi'} & Y'  \ar[r]   & Y' \ar@{=}[d] \\
X' \ar[r]_{P'}          & X' \ar[r]_{F'}           & Y'.}
\end{equation}
Starting from the left-hand side of Equation~\eqref{Id}, we apply
Equation~\eqref{equ:phistarfirst} in the top two rows and get
\[
\xymatrix{
X \ar@{=}[r] \ar@{=}[d]  \ar@{}[dr]|{\eta_{P^*}} & X \ar@{=}[d] \ar[r]^F  & Y \ar@{=}[d] \\
X \ar[r]  \ar[d] \ar@{}[dr]|{\bar{u}^\sharp} & X \ar[r]  \ar[d] \ar@{}[dr]|{\alpha} &  Y  \ar[d]  \\
X' \ar[r]_{P'} & X' \ar[r]_{F'} & Y' .}
\]
Then, Equation~\eqref{equ:transposefirst} gives us
\begin{equation}
\label{equ:firstresult} \xycenter{
X  \ar@{=}[r] \ar[d] & X \ar[r]^F  \ar[d]  \ar@{}[dr]|{\alpha}  & Y \ar[d] \\
X' \ar@{=}[r] \ar@{}[dr]|{\eta_{P'}} \ar@{=}[d] & X' \ar[r] \ar@{=}[d]  & Y' \ar@{=}[d]  \\
X' \ar[r]_{P'} & X' \ar[r]_{F'} & Y' . }
\end{equation}
Considering now the right-hand side of Equation~\eqref{Id}, an application of the analogue of
Equation~\eqref{equ:transposefirst} for $\bar{v}^\sharp$ gives us
\[
\xymatrix{
X \ar[r]^F \ar[d] \ar@{}[dr]|{\alpha} & Y \ar@{=}[r] \ar[d]  & Y \ar[d] \\
X' \ar[r]  \ar@{=}[d] & Y' \ar@{=}[r]  \ar@{=}[d]  \ar@{}[dr]|{\eta_{Q'}} & Y'  \ar@{=}[d] \\
X' \ar[r]   \ar@{=}[d] \ar@{}[drr]|{\phi'} & Y' \ar[r]  & Y'  \ar@{=}[d]  \\
X' \ar[r]_{P'}  & X' \ar[r]_{F'}  & Y'. }
\]
An application of the second axiom for a horizontal monad map for
$(F', \phi')$ then gives us exactly~\eqref{equ:firstresult}, as
required. It remains to show the commutativity of the diagram in~\eqref{equ:diag2},
which amounts to the equation
\begin{equation}
\label{Q+} \xycenter{ P_0 \ar[rr]^E \ar@{=}[d] \ar@{}[drr]|{\theta}
& &
Y \ar@{=}[d] \ar[r]^Q & Y \ar@{=}[d] \\
X  \ar@{=}[d]  \ar[r] & Y \ar@{=}[d] \ar[r]  \ar@{}[drr]|{\nu_{Q^*}} & Y \ar[r] & Y  \ar@{=}[d] \\
X \ar[r]  \ar@{=}[d]   \ar@{}[drrr]|{\phi^*}  & Y \ar[rr]        &   & Y \ar@{=}[d]  \\
X \ar[r]  \ar[d]   \ar@{}[dr]|{\bar{u}^\sharp} & X \ar[d] \ar[rr]
\ar@{}[drr]|{\alpha} & & Y
\ar[d]   \\
X' \ar[r]_{P'}            & X' \ar[rr]_{F'}         &     & Y' }
\quad = \quad
 \xycenter{
 X \ar[rr]^E \ar@{=}[d] \ar@{}[drr]|{\theta} & & Y \ar@{=}[d]
 \ar[r]^Q & Y \ar@{=}[d] \\
 X \ar@{=}[d]  \ar[r]   & Y \ar@{=}[d]  \ar[r]  \ar@{}[drr]|{\nu_{Q^*}}   & Y \ar[r]   & Y \ar@{=}[d]  \\
X \ar[r]   \ar[d]   \ar@{}[dr]|{\alpha}  & Y \ar[rr]   \ar@{}[drr]|{\bar{v}^\sharp}   \ar[d]    &   & Y \ar[d]  \\
X' \ar[r]  \ar@{=}[d]  \ar@{}[drrr]|{\phi'} & Y'  \ar[rr]  &   & Y' \ar@{=}[d] \\
X' \ar[r]_{P'}          & X' \ar[rr]_{F'}   &        & Y' .}
\end{equation}
Starting from the left-hand side of Equation~\eqref{Q+}, we use
Equation~\eqref{equ:phistarsecond} in the second and third row to
get
\[
 \xycenter{
 X \ar[rr]^E \ar@{=}[d] \ar@{}[drr]|{\theta} & &
Y \ar@{=}[d] \ar[r]^Q & Y \ar@{=}[d] \\
X \ar@{=}[d]\ar[r] & Y \ar@{}[d]|{\phi^*}\ar[r]  & Y \ar@{=}[d]\ar[r]  & Y \ar@{=}[d] \\
X \ar@{=}[d]\ar[r] & X \ar@{=}[d]\ar[r] & Y \ar@{}[d]|{\phi}\ar[r] & Y \ar@{=}[d]\\
X \ar@{=}[d]\ar[r]  & X \ar@{}[d]|{\nu_{P^*}}\ar[r]  &
X\ar@{=}[d]\ar[r]  &
Y \ar@{=}[d] \\
X \ar@{}[drr]|{\bar u^\sharp}\ar[d]\ar[rr]  && X
\ar@{}[rd]|{\alpha}\ar[d]\ar[r]  &
Y \ar[d] \\
X' \ar[rr]_{P'} && X' \ar[r]_{F'} & Y'  .}
\]
We then use Equation~\eqref{equ:transposesecond} in the bottom two
rows and obtain
\[
\xymatrix{ X \ar[rr]^E \ar@{=}[d] \ar@{}[drr]|{\theta} & &
Y \ar@{=}[d] \ar[r]^Q & Y \ar@{=}[d] \\
X \ar@{=}[d]\ar[r]  & Y \ar@{}[d]|{\phi^*}\ar[r]  & Y \ar@{=}[d]\ar[r]  & Y \ar@{=}[d] \\
X \ar@{=}[d]\ar[r] & X \ar@{=}[d]\ar[r] & X \ar@{}[d]|{\phi}\ar[r] &
Y \ar@{=}[d]
\\
X \ar@{}[rd]|{\bar{u}^\sharp}  \ar[d]\ar[r]  &  X \ar@{}[rd]|{\bar{u}}\ar[d] \ar[r]
& X \ar@{}[rd]|{\alpha}\ar[d]\ar[r]  & Y \ar[d]  \\
X' \ar@{}[drr]|{\mu_{P'}}\ar@{=}[d] \ar[r]  & X' \ar[r] & X' \ar@{=}[d]\ar[r]  &  Y' \ar@{=}[d] \\
X' \ar[rr]_{P'} && X' \ar[r]_{F'} & Y'  \, . }
\]
We apply Equation~\eqref{equ:hypothesisalpha}, which is the
assumption that $\alpha$ is an endomorphism square, in  the third
and the fourth row, so as to get
\[
\xycenter{ X \ar[rr]^E \ar@{=}[d] \ar@{}[drr]|{\theta} & &
Y \ar@{=}[d] \ar[r]^Q  & Y \ar@{=}[d] \\
X \ar@{=}[d]\ar[r] & Y \ar@{}[d]|{\phi^*}\ar[r] &
Y \ar@{=}[d]\ar[r]  & Y \ar@{=}[d] \\
X \ar@{}[rd]|{\bar u^\sharp}\ar[d]\ar[r] & X
\ar@{}[rd]|{\alpha}\ar[d]\ar[r] & Y
\ar@{}[rd]|{\bar{v}}\ar[d]\ar[r] & Y \ar[d] \\
X' \ar@{=}[d]\ar[r] & X' \ar@{}[rrd]|{\phi'}\ar@{=}[d] \ar[r]  & Y'
\ar[r]  & Y' \ar@{=}[d]
\\
X' \ar@{}[drr]|{\mu_{P'}}\ar@{=}[d]\ar[r]  & X' \ar[r] & X'
\ar@{=}[d]\ar[r]  &
Y' \ar@{=}[d] \\
X' \ar[rr]_{P'} && X' \ar[r]_{F'} & Y' .}
\]
We now apply Equation~\eqref{equ:indhyp} in the top three rows and we obtain
\[
\xycenter{ X \ar[rr]^E \ar@{=}[d] \ar@{}[drr]|{\theta} & &
Y \ar@{=}[d] \ar[r]^Q & Y \ar@{=}[d] \\
 X \ar@{}[rd]|{\alpha} \ar[d]\ar[r] & Y
\ar[d]\ar@{}[rd]|{\bar{v}^\sharp}
\ar[r] & Y \ar@{}[rd]|{\bar v}\ar[d]\ar[r] & Y \ar[d] \\
X' \ar@{}[rrd]|{\phi'}\ar@{=}[d]\ar[r]  & Y' \ar[r]  & Y'  \ar@{=}[d]\ar[r] & Y' \ar@{=}[d] \\
X' \ar@{=}[d]\ar[r] &  X' \ar@{}[rrd]|{\phi'}\ar@{=}[d] \ar[r]  & Y' \ar[r]  &  Y' \ar@{=}[d]  \\
X' \ar@{}[drr]|{\mu_{P'}}\ar@{=}[d]\ar[r]  & X' \ar[r] & X'
\ar@{=}[d]\ar[r]  &
Y' \ar@{=}[d] \\
X' \ar[rr]_{P'} && Y' \ar[r]_{F'} & Y' .}
\]
We use the first axiom for a horizontal monad map (see item (ii) of
Definition~\ref{thm:mnd}) for~$(F', \phi')$ in the bottom three rows
so as to get
\[
\xymatrix{ X \ar[rr]^E \ar@{=}[d] \ar@{}[drr]|{\theta} & &
Y \ar@{=}[d] \ar[r]^Q & Y \ar@{=}[d] \\
X \ar@{}[dr]|{\alpha} \ar[d]\ar[r] & Y
\ar[d]\ar@{}[dr]|{\bar{v}^\sharp} \ar[r]  &
Y \ar@{}[rd]|{\bar v}\ar[d]\ar[r] & Y \ar[d]  \\
X' \ar@{=}[d]\ar[r]  & Y' \ar@{=}[d]\ar@{}[rrd]|{\mu_{Q'}} \ar[r] & Y'
\ar[r] & Y' \ar@{=}[d] \\
X' \ar@{}[rrrd]|{\phi'}\ar@{=}[d]\ar[r] & Y' \ar[rr]  & & Y'
\ar@{=}[d]
\\
X' \ar[r]_{P'} & X' \ar[rr]_{F'} & & Y' . }
\]
We obtain exactly this diagram also by applying the analogue of
Equation~\eqref{equ:transposesecond} for $\bar{v}^\sharp$ to the
second and third row of the right-hand side of Equation~\eqref{Q+}.
This concludes the proof of Theorem~\ref{thm:freemndfrmbicat}. \hfill \qed

\section*{Acknowledgements}

We are grateful to Richard Garner for drawing our attention to Sam Staton's
characterization of free monads. We are also grateful to Mike Shulman for pointing out that an earlier version of this paper contained an inaccuracy regarding fibered functors  and that the proof of  Theorem~\ref{thm:freemndfrmbicat} requires the hypothesis that the source and target functors preserve 
equalizers. We would also like to thank the editor and the referee for the very rapid handling of the
paper.

\medskip

We gratefully acknowledge the support and hospitality of the Centre de Recerca Matem\`atica 
during the academic year 2007/08, in
occasion of the thematic programme on Homotopy Theory and Higher
Categories. Thomas M.~Fiore was supported at the University of
Chicago by NSF Grant DMS-0501208. At the Universitat Aut\`{o}noma de
Barcelona he was supported by grant SB2006-0085 of the Spanish
Ministerio de Educaci\'{o}n y Ciencia under the Programa Nacional de
ayudas para la movilidad de profesores de universidad e
investigadores espa$\tilde{\text{n}}$oles y extranjeros. Thomas M. Fiore
was also supported by the Max Planck Institut f\"ur Mathematik, and
he thanks MPIM for its kind hospitality. Joachim Kock was partially supported by grants 
MTM2006-11391 and MTM2007-63277 of Spain and SGR2005-00606 of Catalonia.

\bibliographystyle{srtnumbered}

\begin{thebibliography}{10}
\newcommand{\enquote}[1]{`#1'}

\bibitem{BarrM:toptt}
M.~Barr and C.~Wells.
\newblock {\em Toposes, Triples, and Theories\/} (Springer-Verlag, 1983).

\bibitem{BenabouJ:intb}
J.~B\'enabou.
\newblock \enquote{Introduction to bicategories}.
\newblock In {\em Reports of the {M}idwest {C}ategory {S}eminar\/}, {\em
  Lecture Notes in Mathematics\/}, Volume~47 (Springer, 1967), 1--77.

\bibitem{BurroniA:tcat}
A.~Burroni.
\newblock \enquote{{$T$}-cat\'egories (cat\'egories dans une triple)}.
\newblock {\em Cahiers de Topologie et G\'eometrie Diffe\'rentielle
  Cat\'egoriques\/} {\bf 12} (1971), 215--321.

\bibitem{EhresmannC:cats}
C.~Ehresmann.
\newblock \enquote{Cat\'egories structur\'ees}.
\newblock {\em Annales Scientifiques de l'\'Ecole Normale Sup\'erieure.
  Troisi\`eme S\'eries\/} {\bf 80} (1963), 349--426.

\bibitem{FioreT:pseapd}
T.~M. Fiore.
\newblock \enquote{Pseudo algebras and pseudo double categories}.
\newblock {\em Journal of Homotopy and Related Structures\/} {\bf 2}~(2)
  (2007), 119--170.

\bibitem{FioreT:modscs}
T.~M. Fiore, S.~Paoli, and D.~Pronk.
\newblock \enquote{Model structures on the category of small double
  categories}.
\newblock {\em Algebraic \& Geometric Topology\/} {\bf 8}~(4) (2008),
  1855--1959.

\bibitem{GambinoN:weltpf}
N.~Gambino and M.~Hyland.
\newblock \enquote{Wellfounded trees and dependent polynomial functors}.
\newblock In S.~Berardi, M.~Coppo, and F.~Damiani, editors, {\em Types for
  proofs and programs\/}, {\em Lecture Notes in Computer Science\/}, Volume
  3085 (Springer, 2004), 210--225.

\bibitem{GambinoN:polfpm}
N.~Gambino and J.~Kock.
\newblock \enquote{Polynomial functors and polynomial monads}.
\newblock ArXiv:0906.4931v2, 2010.

\bibitem{GordonR:coht}
R.~Gordon, A.~J. Power, and R.~Street.
\newblock \enquote{Coherence for tricategories}.
\newblock {\em Memoirs of the American Mathematical Society\/} {\bf 117}~(558)
  (1995), vi+81.

\bibitem{GrandisM:limdc}
M.~Grandis and R.~Par\'e.
\newblock \enquote{Limits in double categories}.
\newblock {\em Cahiers de Topologie et G\'eometrie Diffe\'rentielle
  Cat\'egoriques\/} {\bf 40}~(3) (1999), 162--220.

\bibitem{GrandisM:adjdc}
M.~Grandis and R.~Par\'e.
\newblock \enquote{Adjoint for double categories. {A}ddenda to ``{L}imits in
  double categories''}.
\newblock {\em Cahiers de Topologie et G\'eometrie Diffe\'rentielle
  Cat\'egoriques\/} {\bf 45}~(3) (2004), 193--240.


\bibitem{JacobsB:catltt}
B.~Jacobs.
\newblock {\em Categorical Logic and Type Theory\/} (Elsevier, 1999).

\bibitem{KellyGM:unittc}
G.~M. Kelly.
\newblock \enquote{A unified treatment of transfinite constructions for free
  algebras, free monoids, colimits, associated sheaves, and so on}.
\newblock {\em Bulletin of the Australian Mathematical Society\/} {\bf 22}
  (1980), 1--83.

\bibitem{KellyGM:reve2c}
G.~M. Kelly and R.~H. Street.
\newblock \enquote{Review of the elements of $2$-categories}.
\newblock In {\em Category {S}eminar ({P}roc. {S}em., {S}ydney, 1972/1973)\/},
  {\em Lecture Notes in Mathematics\/}, Volume 420 (Springer, 1974), 75--103.

\bibitem{LackS:fortm}
S.~Lack and R.~Street.
\newblock \enquote{The formal theory of monads {II}}.
\newblock {\em Journal of Pure and Applied Algebra\/} {\bf 175}~(1--3) (2002),
  243--265.

\bibitem{MacLaneS:catwm}
S.~Mac Lane.
\newblock {\em Categories for the working mathematician\/} (Springer, 1998),
  2nd edn.

\bibitem{LeinsterT:higohc}
T.~Leinster.
\newblock {\em Higher operads, higher categories\/} (Cambridge University
  Press, 2004).

\bibitem{MaiettiM:joyaul}
M.~E. Maietti.
\newblock \enquote{Joyal's arithmetic universes as a list-arithmetic pretopos}.
\newblock {\em Theory and Applications of Categories\/} {\bf 24}~(3) (2010),
  39--83.

\bibitem{MoerdijkI:weltc}
I.~Moerdijk and E.~Palmgren.
\newblock \enquote{Wellfounded trees in categories}.
\newblock {\em Annals of Pure and Applied Logic\/} {\bf 104} (2000), 189--218.

\bibitem{NordstronB:promlt}
B.~Nordstrom, K.~Petersson, and J.~Smith.
\newblock {\em Programming in {M}artin-{L}\"of Type Theory. {A}n
  Introduction\/} (Oxford University Press, 1990).

\bibitem{ShulmanM:frabmf}
M.~Shulman.
\newblock \enquote{Framed bicategories and monoidal fibrations}.
\newblock {\em Theory and Applications of Categories\/} {\bf 20}~(18) (2008),
  650--738.

\bibitem{StatonS:namppc}
S.~Staton.
\newblock {\em Name-passing process calculi: operational models and structural
  operational semantics\/}.
\newblock Ph.D. thesis, University of Cambridge.
\newblock Available as Computer Laboratory Technical Report no. 688.
\newblock 2007.

\bibitem{StreetR:fortm}
R.~Street.
\newblock \enquote{The formal theory of monads}.
\newblock {\em Journal of Pure and Applied Algebra\/} {\bf 2}~(2) (1972),
  149--168.

\bibitem{StreetR:fibb}
R.~Street.
\newblock \enquote{Fibrations in bicategories}.
\newblock {\em Cahiers de Topologie et G\'eometrie Diffe\'rentielle
  Cat\'egoriques\/} {\bf 21}~(2) (1980), 111--160.

\end{thebibliography}

\end{document}